\documentclass{amsart}
\usepackage{mathrsfs}
\usepackage{amsfonts}
\usepackage{amsmath}
\usepackage{amssymb}
\usepackage{amsthm}
\usepackage{enumerate}
\usepackage{color}
\usepackage{tikz}
\usetikzlibrary{arrows,shapes,chains}

\usepackage[pdftex,linkcolor=blue,citecolor=blue,backref=page]{hyperref}
\numberwithin{equation}{section}

\allowdisplaybreaks
\theoremstyle{definition}
\newtheorem{theorem}{Theorem}[section]
\newtheorem{definition}[theorem]{Definition}

\newtheorem{lemma}[theorem]{Lemma}

\newtheorem{question}[theorem]{Question}
\newtheorem{proposition}[theorem]{Proposition}
\newtheorem{example}[theorem]{Example}
\newtheorem{construction}[theorem]{Construction}
\newtheorem{remark}[theorem]{Remark}
\newtheorem{expectation}[theorem]{Expectation}
\newcommand{\Supp}{\operatorname{Supp} }

\newcommand{\GK}{\operatorname{GKdim}}

\newcommand{\id}{\operatorname{id}}

\newcommand{\NN}{\mathbb{N}}

\newcommand{\bfone}{\mathbf{1}}
\newcommand{\ZZ}{\mathbb{Z}}

\newcommand{\FF}{\mathbb{F}}
\newcommand{\fm}{\mathfrak{m}}
\newcommand{\FFF}{\mathcal{F}}
\newcommand{\GG}{\mathcal{G}}
\newcommand{\II}{\mathcal{I}}
\newcommand{\PP}{\mathcal{P}}
\newcommand{\QQ}{\mathcal{Q}}
\newcommand{\SO}{\mathcal{S}}
\newcommand{\SG}{\mathbb{S}}

\newcommand{\VV}{\mathcal{V}}
\newcommand{\VSG}{\VV\mathbb{S}}
\newcommand{\WW}{\mathcal{W}}
\newcommand{\XX}{\mathcal{X}}
\newcommand{\x}{\textsf{x}}
\newcommand{\bfk}{\textbf{k}}

\newcommand{\TM}{\operatorname{TM}(\mathcal{X})}
\newcommand{\Tr}{\operatorname{Tr}}
\newcommand{\TX}{\mathcal{T}(\mathcal X)}

\newcommand{\gb}{Gr\"obner basis}

\newcommand{\gsb}{Gr\"obner-Shirshov basis}
\newcommand{\gsbs}{Gr\"obner-Shirshov bases}
\newcommand{\lm}{\overline}
\newcommand{\lt}{\operatorname{lt}}
\newcommand{\lc}{\operatorname{lc}}
\newcommand{\Irr}{\operatorname{Irr}}
\newcommand{\msub}{\operatorname{MaxSub}}
\newcommand{\wt}{\operatorname{wt}}
\newcommand{\h}{\operatorname{h}}
\newcommand{\bran}{\operatorname{bran}}
\newcommand{\Int}{\operatorname{Int}}
\newcommand{\Leaves}{\operatorname{Leaves}}
\newcommand{\Root}{\operatorname{Root}}
\newcommand{\Parent}{\operatorname{Parent}}
\newcommand{\Child}{\operatorname{Child}}
\newcommand{\Ar}{\operatorname{Ar}}
\newcommand{\TIV}{\operatorname{TIV}}
\newcommand{\Twig}{\msub}
\newcommand{\Path}{\operatorname{Path}}
\newcommand{\Piv}{\operatorname{Pivot}}
\renewcommand{\Vert}{\operatorname{Vert}}

\title{Growth of nonsymmetric operads}

\author[Qi]{Zihao Qi}
\address{Department of Mathematics,
Fudan University, Shanghai 200433, China}
\email{qizihao@foxmail.com}

\author[Xu]{Yongjun Xu}
\address{School of Mathematical Sciences, Qufu Normal University,
Qufu 273165, China}
\email{yjxu2002@163.com}

\author[Zhang]{James J. Zhang}
\address{Department of Mathematics, Box 354350,
University of Washington, Seattle, Washington 98195, USA}
\email{zhang@math.washington.edu}

\author[Zhao]{Xiangui Zhao}
\address{School of Mathematics and Statistics,
Huizhou University, Huizhou, Guangdong 516007, China}
\email{zhaoxg@hzu.edu.cn}



\begin{document}

\begin{abstract}
The paper concerns the Gelfand-Kirillov dimension and the generating
series of nonsymmetric operads. An analogue of Bergman's gap theorem
is proved, namely, no finitely generated locally finite nonsymmetric
operad has GK-dimension strictly between $1$ and $2$. For every
$r\in \{0\}\cup \{1\}\cup [2,\infty)$ or $r=\infty$, we construct
a single-element generated nonsymmetric operad $\PP$ such that
$\GK(\PP)=r$. We also provide counterexamples to two expectations of
Khoroshkin and Piontkovski about the generating series of operads.
\end{abstract}

\maketitle


\setcounter{section}{-1}

\section{Introduction}
\label{xxsec0}

Let $\FF$ be a base field. An algebra stands for a unital associative
algebra over $\FF$ unless otherwise stated. The Gelfand-Kirillov
dimension (GK-dimension for short) of an algebra $A$ is defined to be
\[
\GK(A):=\sup_V\; \{ \limsup_{n\to\infty} \;
\log_n [\dim_{\FF}(\sum_{i=0}^{n}V^i)]\}
\]
where the supremum is taken over all finite dimensional subspaces $V$
of $A$. Similarly, one can define the GK-dimension of nonassociative
algebras. GK-dimension is a standard and powerful invariant for
investigating associative and nonassociative algebras. We refer to
\cite{KL00} for more background and properties of the GK-dimension of
algebras and modules.

The range of possible values for the GK-dimension of an algebra is
\begin{equation}
\label{E0.0.1}\tag{E0.0.1}
R_{\GK}:=\{0\}\cup\{1\}\cup[2,\infty) \cup \{\infty\}.
\end{equation}
The gap between $0$ and $1$ follows easily from the definition of
GK-dimension. The existence problem of algebras $A$ with $1<\GK(A)<2$
was open for some years until Bergman \cite[Theorem 2.5]{KL00}
proved that no such algebras exist. Bergman's gap theorem is also
valid for some other classes of algebras, for example, Jordan algebras
\cite{MZ96} and dialgebras \cite{zhang2019no}. However, there exist
Lie algebras \cite{Pet97} and Jordan superalgebras \cite{PS19} with
GK-dimension strictly between $1$ and $2$.

The notion of an operad was first introduced by Boardman-Vogt
\cite{BV73} and May \cite{May72} in late 1960s and early 1970s in
the study of iterated loop spaces. Since 1990s, due to
Ginzburg-Kapranov's Koszul duality theory of operads \cite{GK94},
Kontsevich's \cite{Kon03} and Tamarkin's \cite{Tam98} operadic
approach to the formality theorem, as well as Getzler's study on
topological field theories \cite{Get94, Get95}, operad theory has
become an important tool in homological algebra, category theory,
algebraic geometry and mathematical physics.

In this paper, we investigate the GK-dimension and the generating
series of nonsymmetric operads. The GK-dimension of locally finite
operads was defined in \cite[p. 400]{KP15} and
\cite[Definition 4.1]{BYZ20}. Let $\PP$ be a \emph{locally finite}
operad, i.e., an operad $\PP$ with each $\PP(i)$ finite dimensional
over the base field $\FF$. The \emph{GK-dimension} of $\PP$ is
defined to be
\[
\GK(\PP):=\limsup_{n\to\infty} \;
\log_n\left(\sum_{i=0}^n\dim_{\FF} \PP(i)\right).
\]
Our main result is an analogue of Bergman's gap theorem:

\begin{theorem}
\label{xxthm0.1}
No finitely generated locally finite nonsymmetric operad
has GK-dimension strictly between $1$ and $2$.
\end{theorem}

For the definition of a finitely generated operad, see Definition
\ref{xxdef1.16}. Bergman's gap theorem for associative
algebras was proved by counting specific words that satisfy a set
of conditions and form a monomial basis of the algebra under
consideration, equivalently, by counting a set of
\emph{single-branched tree monomials} (defined in Section
\ref{xxsec4}) where the algebra is interpreted as an operad.
This method cannot be extended directly to prove the gap
theorem for nonsymmetric operads
since the tree monomials we want to count are not necessarily
single-branched. To overcome the above difficulty, we divide the
underlying tree into three subtrees such that the ``big'' subtree
is single-branched and thus we can count its tree monomials
similarly as for the case of associative algebras. The other two
subtrees are both ``small'' such that their tree monomials are well
controlled.

Note that if we drop the condition ``finitely generated'' in
Theorem \ref{xxthm0.1}, then the statement does not hold. In
fact, as Example \ref{xxex3.4} shows, any positive real number
is the GK-dimension of some nonsymmetric operad.

Theorem \ref{xxthm0.1} serves as an essential ingredient of the
next result. Recall that $R_{\GK}$ is defined in \eqref{E0.0.1}.

\begin{theorem}
\label{xxthm0.2}
\begin{enumerate}
\item[(1)]
If $\PP$ is a finitely generated locally finite nonsymmetric
operad, then $\GK(\PP)\in R_{\GK}$.
\item[(2)]
If $r\in R_{\GK}$, then there is a single-element generated
single-branched locally finite nonsymmetric operad $\PP$
such that $\GK(\PP)=r$.
\end{enumerate}
\end{theorem}

The Gelfand-Kirillov dimension of an operad is closely related
to the generating series that is defined as follows. Let $\PP$
be a locally finite operad. The {\it generating series} of
$\PP$ is defined to be the formal power series \cite[(0.1.2)]{KP15}
\begin{equation}
\label{E0.2.1}\tag{E0.2.1}
G_{\PP}(z):=\sum_{n=0}^{\infty} \dim_{\FF} \PP(n) \; z^n.
\end{equation}

We recall the following definition.

\begin{definition} \cite{Sta80, Zei90, BBY12, Ber14, KP15}
\label{xxdef0.3}
Let $F(z):=\sum_{n\geq 0} f(n) z^n$ be a formal power series or
a $C^{\infty}$-function where $f(n)\in {\mathbb R}$ for all $n$.
\begin{enumerate}
\item[(1)]
$F(z)$ is called {\it holonomic} (also called {\it $D$-finite}
or {\it differentiably finite}) if it satisfies a nontrivial
linear differential equation with polynomial coefficients.
\item[(2)]
$F(z)$ is called {\it differential algebraic} if it satisfies
a nontrivial algebraic differential equation with polynomial
coefficients.
\end{enumerate}
\end{definition}

It is well-known that
\begin{equation}
\notag
{\text{rational }}
\Longrightarrow
{\text{ algebraic (over ${\mathbb R}(z)$)}}
\Longrightarrow
{\text{ holonomic}}\Longrightarrow
{\text{differential algebraic}}
\end{equation}
where the second implication is \cite[Theorem 2.1]{Sta80}.
Several researchers have recently been studying holonomic
and differential algebraic property of $G_{\PP}(z)$
\cite{Ber14, KP15}. In particular, Khoroshkin-Piontkovski
showed that, under moderate assumptions, operads with a finite
Gr{\" o}bner basis have rational, or algebraic, or differential
algebraic generating series \cite[Theorems 0.1.3, 0.1.4, and 0.1.5]{KP15}.
In \cite[Section 4]{KP15} Khoroshkin-Piontkovski listed
several expectations and conjectures, one of which is

\begin{expectation} \cite[Expectation 2]{KP15}
\label{xxexp0.4}
The generating series of a generic finitely presented
nonsymmetric operads is algebraic over ${\mathbb Z}[z]$.
\end{expectation}

We construct a finitely presented nonsymmetric operad
such that the generating series is not holonomic (hence, not
algebraic), which provides a (non-generic) counterexample to the
above Expectation [Example \ref{xxex6.4}]. We also prove

\begin{proposition}
\label{xxpro0.5} Let $r\in R_{\GK}\setminus\{0\}$. Then there
is a single-element generated locally finite nonsymmetric operad
$\PP$ with $\GK(\PP)=r$ and $G_{\PP}(z)$ not being holonomic.
As a consequence, $G_{\PP}(z)$ is neither rational nor algebraic
in this case. Therefore such an operad is a counterexample to
\cite[Expectation 3]{KP15}.
\end{proposition}

\begin{remark}
\label{xxrem0.6}
Constructions \ref{xxcon1.3}, \ref{xxcon5.1} and \ref{xxcon7.1}
provide useful constructions of nonsymmetric operads
(or symmetric ones in Construction \ref{xxcon7.1}) from
graded algebras (or monomial algebras in Construction
\ref{xxcon5.1}). A lot of algebraic properties of graded (or
monomial) algebras can be transformed to the corresponding
properties of the related operads.
Using this idea, in addition to Proposition \ref{xxpro0.5} and
Example \ref{xxex6.4}, we obtain a potential counterexample to
\cite[Expectation 1]{KP15} in Example \ref{xxex7.5}.
\end{remark}

This paper mainly concerns nonsymmetric operads.
In the final section we touch upon the symmetric ones. The
following is easy to prove.

\begin{theorem}
\label{xxthm0.7}
\begin{enumerate}
\item[(1)]
Let $\PP$ be a finitely generated locally finite symmetric
operad. Then $\GK(\PP)\in R_{\GK} \cup (1,2)$.
\item[(2)]
For every $r\in R_{\GK}\setminus (2,3)$, there exists a
finitely generated locally finite symmetric operad $\PP$
such that $\GK(\PP)=r$.
\item[(3)]
For every $r\in R_{\GK}\setminus (\{0\} \cup \{1\} \cup (2,3))$,
there exists a finitely generated locally finite symmetric operad
$\PP$ such that $\GK(\PP)=r$ and that $G_{\PP}(z)$ is not
holonomic.
\end{enumerate}
\end{theorem}

In light of \cite{KP15, Ber14}, it would be very interesting
to determine which classes of operads have rational (resp.
algebraic, holonomic, differential algebraic) generating series.
Theorem \ref{xxthm0.7} suggests that, generically, $G_{\PP}(z)$ is not
holonomic. Theorem \ref{xxthm0.7}(1,2) lead to the following question.

\begin{question}
\label{xxque0.8}
Let $r\in (1,2) \cup (2,3)$. Is there a finitely generated
locally finite symmetric operad $\PP$ such that
$\GK(\PP)=r$?
\end{question}

In view of Theorem \ref{xxthm0.7}(2,3) the following result of
\cite{BYZ20} is quite surprising, see \cite{BYZ20} for
details.

\begin{theorem}
\label{xxthm0.10}
Let $\PP$ be a 2-unitary locally finite symmetric operad. Then
$\GK(\PP)$ is either an integer or the infinity. If $\GK(\PP)$
is finite, then $\PP$ is finitely generated and $G_{\PP}(z)$
is rational.
\end{theorem}

This paper is organized as follows. Section \ref{xxsec1}
recalls basic definitions and properties of nonsymmetric operads.
Section \ref{xxsec2} introduces Gr\"obner-Shirshov bases of
nonsymmetric operads. Section \ref{xxsec3} contains definitions,
examples and properties of the GK-dimension of nonsymmetric
operads. In Section \ref{xxsec4}, we prove the main result,
Theorem \ref{xxthm0.1}, and in Section \ref{xxsec5}, we prove
Theorem \ref{xxthm0.2}. In Section \ref{xxsec6} we study the
generating series of operads and prove Proposition \ref{xxpro0.5}.
Finally, Section \ref{xxsec7} provides some comments and
examples about symmetric operads.

\section{Preliminaries}
\label{xxsec1}
Let $\FF$ be the base field and $\FF^*:=\FF\setminus\{0\}$.
Let $\NN$ denote the natural numbers and $\NN^*:=\NN\setminus\{0\}$.
To save space, some non-essential details are omitted here
and there. But we try to provide as much detail as possible
for the proof of the main result -- Theorem \ref{xxthm0.1}.

\subsection{Ns operads}
\label{xxsec1.1}
The following definition is copied from
\cite[Definition 3.2.2.3]{BD16} (see also \cite[Chapter 5]{LV12}).

\begin{definition}
\label{xxdef1.1}
A \emph{nonsymmetric operad} or simply \emph{ns operad} is
a collection of vector spaces $\PP=\{\PP(n)\}_{n\geq0}$
equipped with an element $\id\in\PP(1)$ (called the
\emph{identity element}) and maps (called \emph{partial
compositions}), for $1\leq i\leq n$,
\[
\circ_i:\PP(n)\otimes \PP(m)\to \PP(n+m-1),\quad
\alpha\otimes \beta\mapsto\alpha\circ_i\beta,
\]
which satisfy the following properties for all $\alpha\in\PP(n)$,
$\beta\in\PP(m)$ and $\gamma\in \PP(r)$:
\begin{enumerate}[(i)]
\item
(sequential axiom)
\begin{align}
\label{E1.1.1}\tag{E1.1.1}
(\alpha\circ_i\beta)\circ_j\gamma=
\alpha\circ_i(\beta\circ_{j-i+1}\gamma)\ \text{ for } i\leq j\leq i+m-1;
\end{align}
\item
(parallel axiom)
\begin{align}
\label{E1.1.2}\tag{E1.1.2}
(\alpha\circ_i\beta)\circ_j\gamma=
\begin{cases}
(\alpha\circ_{j-m+1}\gamma)\circ_i\beta,& i+m\leq j\leq n+m-1,\\
(\alpha\circ_{j}\gamma)\circ_{i+r-1}\beta,& 1\leq j\leq i-1;
\end{cases}
\end{align}
\item
(unit axiom)
\begin{align}
\label{E1.1.3}\tag{E1.1.3}
\id\circ_1\alpha=\alpha,\quad
\alpha\circ_i\id=\alpha\ \text{ for }1\leq i\leq n.
\end{align}
\end{enumerate}
\end{definition}

The above definition is called the \emph{partial definition
of a ns operad}. For the classical definition and the
monoidal definition, the interested readers are referred to
\cite[Chapter 5]{LV12}.
Except for the final section we only consider ns operads, and
sometimes ``ns'' or the word ``nonsymmetric'' is omitted.

A collection $\PP=\{\PP(n)\}_{n\geq0}$ of spaces (especially,
an operad) is called \emph{finite dimensional} if the
\emph{dimension of $\PP$} is finite, i.e.,
$\dim\PP:=\dim\left(\oplus_{n\geq0}\PP(n)\right)<\infty$;
$\PP$ is called \emph{locally finite}
(respectively, \emph{reduced}, \emph{connected}) if
$\PP(n)$ is finite dimensional for all $n\in\NN$
(respectively, if $\PP(0)=0$, if $\PP(1)=\FF\id\cong \FF$).
We say a collection $\VV=\{\VV(n)\}_{n\geq0}$ of spaces
is a \emph{subcollection} of $\PP$ if $\VV(n)$ is a
subspace of $\PP(n)$ for all $n\geq0$. Furthermore, if
the subcollection $\VV$ is an operad with the partial
compositions of $\PP$, we call $\VV$ a \emph{suboperad}
of $\PP$.

To simplify notation, we also view a collection
$\GG=\{\GG(n)\}_{n\geq 0}$ of sets (e.g., a collection
of vector spaces) as a disjoint union $\GG=\sqcup_{n\geq0}\GG(n)$
(respectively, a direct sum $\GG=\oplus_{n\geq 0}\GG(n)$),
and vice versa.

As demonstrated in the following examples, an operad
can be viewed as a generalization of an algebra.

\begin{example} \cite[p.137]{LV12}
\label{xxex1.2}
A unital associative algebra $A$ can be interpreted as
an operad $\PP$ with $\PP(1)=A$ and $\PP(n)=0$ for all
$n\neq1$, and the compositions in $\PP$ are given by
the multiplication of $A$.
\end{example}

The following construction is due to Dotsenko
\cite{Dot19}.

\begin{construction} \cite[Definition 3.2(2)]{Dot19}.
\label{xxcon1.3}
Let $A:=\oplus_{i\geq0} A_i$ be an $\NN$-graded algebra with unit
$1_A$. Suppose that $A$ has a graded augmentation $\epsilon:
A\to \FF$ such that $\fm:=\ker \epsilon$ is a maximal graded
ideal of $A$. Let $\PP_A(0)=0$ and $\PP_A(n)=A_{n-1}$ for all
$n\geq1$. Define compositions as follows
$$\begin{aligned}
\circ_i:\quad &\PP_A(m)\otimes \PP_A(n)\to \PP_A(n+m-1),\\
& \quad a_{m-1}\otimes a_{n-1}\mapsto \quad
\begin{cases}
    c a_{m-1} &a_{n-1}=c 1_A, c\in \FF,\\
    a_{m-1}a_{n-1} & a_{n-1}\in \fm, i=1,\\
    0 & a_{n-1}\in \fm, i\neq 1.
\end{cases}
\end{aligned}
$$
It is easy to check by definition that
$\PP_A:=\{\PP_A(n)\}_{n\geq 0}$ is an operad with $\id:=1_A$.
This operad is called the {\it min-envelope operad} of $A$.
Piontkovski \cite[Theorem 3.1]{Pio17} used this construction to
produce a counterexample to \cite[Conjecture 10.4.1.1]{BD16}.
See \cite[Definition 3.2(1)]{Dot19} for another related construction.
\end{construction}

If $A$ is generated by $A_1$, then $\PP_A$ is generated by
$\PP_A(2)$. By Lemma \ref{xxlem6.3}, $A$ is finitely
generated as an algebra if and only if $\PP_A$ is finitely
generated as an operad.

For an $\NN$-graded locally finite algebra $A$, its Hilbert
series is defined to be
\begin{equation}
\label{E1.3.1}\tag{E1.3.1}
H_A(z)=\sum_{n=0}^{\infty} \dim A_n z^n
\end{equation}
in the same way of defining the generating series of
an operad.

By the construction given in Construction \ref{xxcon1.3},
one sees that
\begin{equation}
\label{E1.3.2}\tag{E1.3.2}
G_{\PP_A}(z)=zH_A(z).
\end{equation}

A \emph{morphism} $\phi$ between two operads $\PP$ and
$\PP'$ is a collection of linear maps
$\phi_n: \PP(n)\to\PP'(n)$, $n\geq0$, such that
$\phi_1(\id_{\PP})=\id_{\PP'}$ and
\[
\phi_m(u)\circ_i\phi_n(v)=\phi_{m+n-1}(u\circ_i v)
\]
for all $u\in\PP(m)$, $v\in\PP(n)$, $1\leq i\leq m$, $n\geq 0$.
If each $\PP(n)$ is a subspace of $\PP'(n)$, we call
$\PP$ a \emph{subcollection} of $\PP'$, and write
$\PP\subseteq\PP'$.

An \emph{operation alphabet} (or \emph{generating operations})
is a collection $\XX=\{\XX(n)\}_{n\geq0}$ of sets $\XX(n)$.
The number $n$ is called the \emph{arity} of an element
$x\in \XX(n)$ and denoted by $\Ar(x)=n$.

Let $\XX$ be an operation alphabet. The \emph{free operad}
over $\XX$ is an operad $\FFF(\XX)$ equipped with an
inclusion $\eta:\XX\to \FFF(\XX)$ (i.e., a collection of
inclusions $\eta_n:\XX(n)\to \FFF(\XX)(n)$ for all $n\geq 0$)
which satisfies the following universal property:
any map $f:\XX\to \PP$, where $\PP$ is an operad,
extends uniquely to an operad morphism $\tilde{f}:
\FFF(\XX)\to \PP$ \cite[Section 3.3]{BD16}.

\subsection{Planar rooted tree}
\label{xxsec1.2}
Similar to the case of associative algebras, an operad
can be presented by generators and relations, i.e.,
as a quotient of a free operad. In this subsection, we
introduce a language for working with planar rooted trees,
which will be used later to construct free operads.
We mainly follow the ideas in \cite[Section 3.3]{BD16}
and use the language introduced in \cite{BD16} with minor
modification for the rest of the paper, see Remark
\ref{xxrem1.6}.

\begin{definition}
\label{xxdef1.4}
A \emph{rooted tree} ~$\tau$ consists of
\begin{itemize}
\item
a finite set $\Vert(\tau)$ of \emph{vertices},
which is a disjoint union
\[
\Vert(\tau)=\Int(\tau)\sqcup \Leaves(\tau)\sqcup \{r\},
\]
where elements of the (possible empty) set $\Int(\tau)$
are called \emph{internal vertices} of $\tau$, elements
of the nonempty set $\Leaves(\tau)$ are called \emph{leaves}
of $\tau$, and the element $r$ is called the \emph{root}
of $\tau$ and denoted by $\Root(\tau)$; and
\item
a \emph{parent function}
\[
\Parent=\Parent_{\tau}:\Vert(\tau)\setminus\{r\}\to
\Vert(\tau)\setminus \Leaves(\tau)
\]
such that
$|\Parent^{-1}(r)|=1$, {$|\Parent^{-1}(v)|\geq 1$ for
all $v\in\Int(\tau)$,} and the \emph{connectivity condition}
is satisfied: for each vertex $v\neq r$, there are an
$h\in\NN^*$ and vertices
$v_0=r, v_1,\dots, v_{h-1}, v_h=v$,
such that $v_i=\Parent(v_{i+1})$ for $0\leq i\leq h-1$.
\end{itemize}
\end{definition}

When we draw a rooted tree on the plane, the root is
always drawn at the bottom of the tree; a hollow circle
presents an internal vertex while a black solid circle
presents a leaf or the root; a solid segment between two
vertices indicates that the lower vertex is the parent of
the higher one. See Figure \ref{xxfig1-xxrem1.6} for examples.

\begin{figure}[h]
\centering
\begin{tikzpicture}
  \tikzstyle{every node}=[draw,thick,circle,fill=white,minimum size=6pt, inner sep=1pt]
    \node (r) at (0,0) [fill=black,minimum size=0pt,label=right:$r$] {};
    \node (1) at (0,1) [fill=black,minimum size=0pt] {};
    \draw  (r) -- (1);
\node[minimum size=0pt,inner sep=0pt,label=below:$\tau_0$] (name) at (0,-0.5){};
\end{tikzpicture}
\hspace{20mm}
\begin{tikzpicture}[scale=.6]
  \tikzstyle{every node}=[draw,thick,circle,fill=white,minimum size=6pt, inner sep=1pt]
    \node (r) at (0,0) [fill=black,minimum size=0pt,label=right:$r$] {};
    \node (1) at (0,1) {};
    \node (2) at (0,2) {};
    \node (3) at (0,3) [fill=black,minimum size=0pt] {};
    \draw  (r) -- (1)--(2)--(3);
\node[minimum size=0pt,inner sep=0pt,label=below:$\tau_1$] (name) at (0,-0.5){};
\end{tikzpicture}
\hspace{20mm}
\begin{tikzpicture}[scale=0.6]
  \tikzstyle{every node}=[draw,thick,circle,fill=white,minimum size=6pt, inner sep=1pt]
    \node (r) at (0,0) [fill=black,minimum size=0pt,label=right:$r$] {};
    \node (1) at (0,1)  {};
    \node (2a) at (-1,2) [fill=black,minimum size=0pt] {};
    \node (2b) at (1,2) {};
    \node (3a) at (0,3)  [fill=black,minimum size=0pt] {};
    \node (3b) at (1,3)  [fill=black,minimum size=0pt] {};
    \node (3c) at (2,3)  [fill=black,minimum size=0pt] {};
    \draw  (r) -- (1)--(2a);
    \draw         (1)--(2b)--(3a);
    \draw              (2b)--(3b);
    \draw              (2b)--(3c);
\node[minimum size=0pt,inner sep=0pt,label=below:$\tau_2$] (name) at (0,-0.5){};
\end{tikzpicture}
\caption{Rooted trees}\label{xxfig1-xxrem1.6}
\end{figure}
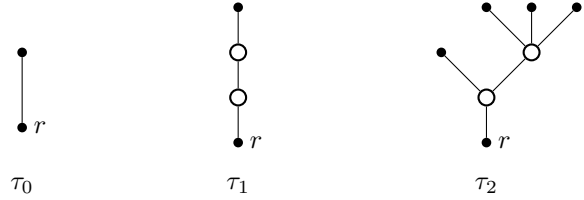

Let $\tau$ be a rooted tree. The following lemma is
straightforward.

\begin{lemma}
\label{xxlem1.5}
Suppose $v\in \Vert(\tau)\setminus\Root(\tau)$, and
$h$ and $v_i$ for $0\leq i\leq h$ are the same as in
Definition \ref{xxdef1.4}. Then the number $h$ and the
sequence $v_0,v_1,\dots,v_h$ are uniquely determined by
$v$, called respectively the \emph{height} of $v$
(denoted by $\h(v)$) and the \emph{path from root to $v$}.
\end{lemma}

The \emph{height} of $\tau$ is the maximal height of its
internal vertices, i.e.,
\[
\h(\tau):=\max\{\h(v):v\in\Int(\tau)\}.
\]
For $v\in \Vert(\tau)\setminus\Leaves(\tau)$,
elements in $\Child(v):=\Parent^{-1}(v)$ are called
\emph{children} of $v$. More generally we define the
\emph{descendants} of $v$ to be
\[
\Child^{\infty}(v):=\{v'\in\Vert(\tau):\Parent^i(v')=v
\text{ for some }i\in\NN^*\}.
\]
The number of children of $v$ (respectively, the number
of leaves of $\tau$) is called the \emph{arity} of $v$
(respectively, of $\tau$), denoted by $\Ar(v)$
(respectively, $\Ar(\tau)$). Define the \emph{weight} of
$\tau$ to be $\wt(\tau):=|\Int(\tau)|$. The only tree with
weight $0$ is called the \emph{trivial (rooted) tree},
denoted by $\tau_0$, the only two vertices of which are
the root $\Root(\tau_0)$ and the leaf in $\Child(\Root(\tau_0))$,
see Figure \ref{xxfig1-xxrem1.6}. For a nontrivial rooted tree,
the only child of the root must be an internal vertex.

When working with more than one
rooted trees, we usually use subscripts (e.g., the height
$\h_{\tau}(v)$) to indicate the tree under consideration.
We write a map $\varphi:\Vert(\tau)\to \Vert(\tau')$ between
the vertex sets of rooted trees $\tau$ and $\tau'$ as
$\varphi:\tau\to \tau'$ for short. An \emph{isomorphism} from
$\tau$ to $\tau'$ is a bijective map $\varphi:\tau\to \tau'$
that respects the parent functions, i.e.,
\[
\varphi(\Parent_{\tau}(v))=\Parent_{\tau'}(\varphi(v)),
\ \forall v\in\Vert(\tau)\setminus\Root(\tau).
\]
We say $\tau$ is \emph{isomorphic} to $\tau'$ if there is an
isomorphism from $\tau$ to $\tau'$. It is easy to see that
an isomorphism also respects the type of each vertex, more
precisely, if $\varphi:\tau\to\tau'$ is an isomorphism, then
$v\in\Vert(\tau)$ is a leaf (respectively, an internal vertex,
the root) in $\tau$ if and only if so is $\varphi(v)$ in $\tau'$.

\begin{remark}
\label{xxrem1.6}
The only difference between our rooted trees and those
in \cite[Section 3.3]{BD16} is that an internal vertex
of our rooted tree always has positive arity while an
internal vertex in the sense of \cite{BD16} may have zero
arity. Note that the existence of internal vertices of
arity zero makes \cite[Proposition 3.4.1.6]{BD16} false.
(see Example \ref{xxex2.3} for a counterexample).
\end{remark}

\begin{definition}
\label{xxdef1.7}
A \emph{planar rooted tree} (\emph{PRT}, for short) is a
rooted tree together with a planar structure, i.e., a rooted tree with a total
order on $\Child(v)$ for each $v\in\Vert(\tau)\setminus
\Leaves(\tau)$.
\end{definition}

The planar structure of a PRT $\tau$ induces a total order
on $\Vert(\tau)$: Suppose $u$ and $u'$ are two different
vertices and consider the paths from root to $u$ and $u'$:
\[
v_0=r,v_1,\dots,v_{\h(u)};\
v'_0=r,v'_1,\dots,v'_{\h(u')}.
\]
We say $u<u'$ if one of the following holds:
\begin{enumerate}[(i)]
\item
$\h(u)<\h(u')$ and $v_0=v'_0,v_1=v'_1,
\dots,v_{\h(u)}=v'_{\h(u)}$;
\item
there is $1\leq k\leq \min\{\h(u),\h(u')\} -1$ such that
$v_0=v'_0,v_1=v'_1,\dots,v_k=v'_k$ and $v_{k+1}< v'_{k+1}$.
\end{enumerate}
If there is no confusion, we usually use positive integers
to denote the leaves, i.e., $\Leaves(\tau)=\{1,2,\dots,\Ar(\tau)\}$,
to indicate the order on $\Leaves(\tau)$ in the obvious way.
We draw a PRT on the plane in a way that the planar order
on $\Child(v)$ is determined by ordering the corresponding
vertices left-to-right. A PRT $\tau$ is \emph{isomorphic} to
PRT $\tau'$ if there exists a rooted tree isomorphism
$\varphi:\tau\to \tau'$ such that $\varphi(u)<\varphi(v)$
whenever $u<v$ for $u,v\in \Vert(\tau)$.

\begin{example}
\label{xxex1.8}
Let $\tau$ and $\tau'$ be PRTs as drawn in Figure \ref{xxfig2-xxex1.8}.
It is easy to see that $\tau$ is isomorphic to $\tau'$ as
rooted trees. However, they are not isomorphic as PRTs.
\end{example}

\begin{figure}[h]
\centering
\begin{tikzpicture}[scale=0.6]
  \tikzstyle{every node}=[draw,thick,circle,fill=white,minimum size=6pt, inner sep=1pt]
    \node (r) at (0,0) [fill=black,minimum size=0pt,label=right:$r$] {};
    \node (1) at (0,1)  {};
    \node (2a) at (-1,2) {};
    \node (2b) at (1,2) [fill=black,minimum size=0pt,label=above:$3$]{};
    \node (3a) at (-2,3)  [fill=black,minimum size=0pt,label=above:$1$]{};
    \node (3b) at (-0,3)  [fill=black,minimum size=0pt,label=above:$2$]{};
    \draw  (r) -- (1)--(2a)--(3a);
    \draw         (1)--(2b);
    \draw              (2a)--(3b);
\node[minimum size=0pt,inner sep=0pt,label=below:$\tau$] (name) at (0,-0.5){};
\end{tikzpicture}
\hspace{20mm}
\begin{tikzpicture}[scale=0.6]
  \tikzstyle{every node}=[draw,thick,circle,fill=white,minimum size=6pt, inner sep=1pt]
    \node (r) at (0,0) [fill=black,minimum size=0pt,label=right:$r'$] {};
    \node (1) at (0,1)  {};
    \node (2a) at (-1,2) [fill=black,minimum size=0pt,label=above:$1$] {};
    \node (2b) at (1,2) {};
    \node (3a) at (0,3)  [fill=black,minimum size=0pt,label=above:$2$] {};
    \node (3b) at (2,3)  [fill=black,minimum size=0pt,label=above:$3$] {};
    \draw  (r) -- (1)--(2a);
    \draw         (1)--(2b)--(3a);
    \draw              (2b)--(3b);
\node[minimum size=0pt,inner sep=0pt,label=below:$\tau'$] (name) at (0,-0.5){};
\end{tikzpicture}
\caption{Non-isomorphic PRTs that are isomorphic as
rooted trees}\label{xxfig2-xxex1.8}
\end{figure}
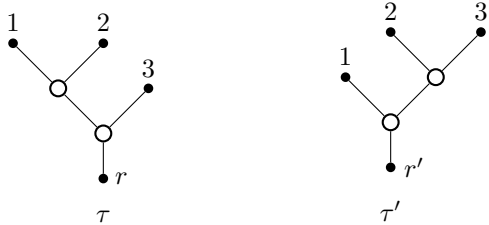

The following lemma is clear.

\begin{lemma} \cite[Definition 3.4.2.1]{BD16}
\label{xxlem1.9}
Let $\tau$ be a PRT. Suppose that $\emptyset\neq V'
\subseteq \Int(\tau)$ satisfies the following conditions
\begin{enumerate}[(i)]
\item
there is a unique $v'\in V'$ such that
$\Parent_{\tau}(v')\not\in V'$; and
\item
for each $v''\in V'$ there exist $h\in\NN^*$ and vertices
$v_1=v'$, $v_2,\dots,v_{h-1}, v_h=v''$ such that
$v_i=\Parent_{\tau}(v_{i+1})$ for all $1\leq i\leq h-1$.
\end{enumerate}
Then $V'$ defines a PRT $\tau'$ as follows
\begin{align*}
\Root(\tau')=\Parent_{\tau}(v'),\
\Int(\tau')=V',\
\Leaves(\tau')=\left(\bigcup_{v\in V'}\Child(v)\right)\setminus V',
\end{align*}
and the parent function and the planar structure of
$\tau'$ are the restrictions of the parent function and
the planar structure of $\tau$. We call $\tau'$ a
\emph{subtree} of $\tau$, denote $\tau'\subseteq \tau$.
\end{lemma}

Note that, in Lemma \ref{xxlem1.9}, condition (i) implies
that each $v_i$ in condition (ii) belongs to $V'$.

Let $\tau$ be a PRT, $v\in \Int(\tau)$, and
$V'=\left(\Int(\tau)\cap\Child^{\infty}(v)\right)\cup\{v\}$.
Then $V'$ satisfies the conditions in Lemma \ref{xxlem1.9}
and thus defines a subtree $\tau'$ of $\tau$. We call this
subtree the \emph{maximal subtree} of $\tau$ rooted at
$\Parent(v)$ and containing internal vertex $v$. Given
$r'\in \Vert(\tau)$ and $v'\in\Int(\tau)\cap {\Child(r')}$,
it is easy to see that there is a unique maximal subtree of
$\tau$ that is rooted at $r'$ and contains $v'$. We denote
it by $\msub_{\tau}(r',v')$.

\subsection{Free ns operads}
\label{xxsec1.3}
Let $\XX$ be an operation alphabet. A \emph{labelling} of
a nontrivial PRT $\tau$ is a map $\x:\Int(\tau)\to\XX$ such
that $\Ar(\x(v))=\Ar(v)=|\Child(v)|$ for all $v\in\Int(\tau)$.
A \emph{nontrivial tree monomial} in $\XX$ is a pair
$T=(\tau,\x)$ of a nontrivial PRT $\tau$ and a labelling $\x$
of $\tau$. Define the tree monomial for the trivial tree to
be the \emph{trivial tree monomial}, denoted by
$1:=(\tau_0,\emptyset)$. The PRT $\tau$ is called the
\emph{underlying tree} of tree monomial $T$, denoted by
$\Tr(T):=\tau$. More generally, if $W$ is a set of tree
monomials, denote $\Tr(W):=\{\tau:\tau=\Tr(T), T\in W\}$.
The \emph{arity} and \emph{weight} of a tree monomial
$T=(\tau,\x)$ are defined as $\Ar(T):=\Ar(\tau)$ and
$\wt(T):=\wt(\tau)$ respectively.
Let $\TM$ denote the set of tree monomials in $\XX$ together
with the trivial tree monomial.

\begin{example}
\label{xxex1.10}
Suppose that $\XX(1)=\{a\}$ and $\XX(2)=\{b,c\}$.
In Figure \ref{xxfig3-xxex1.10} are examples of tree
monomials in $\XX$.
\end{example}

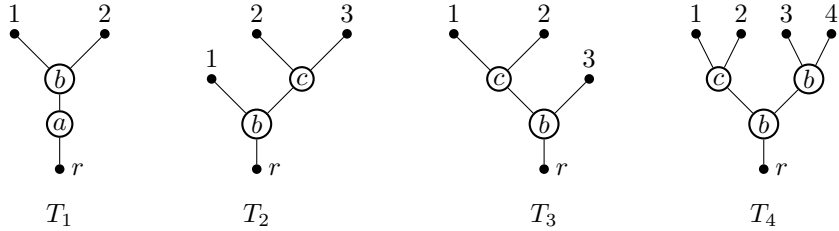
\begin{figure}[h]
\centering
\begin{tikzpicture}[scale=0.6]
\tikzstyle{every node}=[draw,thick,circle,fill=white,minimum size=6pt, inner sep=1pt]
    \node (r) at (0,0) [fill=black,minimum size=0pt,label=right:$r$] {};
    \node (1) at (0,1)  {$a$};
    \node (2) at (0,2) {$b$};
    \node (3a) at (-1,3)  [fill=black,minimum size=0pt,label=above:$1$]{};
    \node (3b) at (1,3)  [fill=black,minimum size=0pt,label=above:$2$]{};
    \draw  (r) -- (1)--(2)--(3a);
    \draw              (2)--(3b);
\node[minimum size=0pt,inner sep=0pt,label=below:$T_1$] (name) at (0,-0.5){};
\end{tikzpicture}
\hspace{8mm}
\begin{tikzpicture}[scale=0.6]
  \tikzstyle{every node}=[draw,thick,circle,fill=white,minimum size=6pt, inner sep=1pt]
    \node (r) at (0,0) [fill=black,minimum size=0pt,label=right:$r$] {};
    \node (1) at (0,1)  {$b$};
    \node (2a) at (-1,2) [fill=black,minimum size=0pt,label=above:$1$] {};
    \node (2b) at (1,2) {$c$};
    \node (3a) at (0,3)  [fill=black,minimum size=0pt,label=above:$2$] {};
    \node (3b) at (2,3)  [fill=black,minimum size=0pt,label=above:$3$] {};
    \draw  (r) -- (1)--(2a);
    \draw         (1)--(2b)--(3a);
    \draw              (2b)--(3b);
\node[minimum size=0pt,inner sep=0pt,label=below:$T_2$] (name) at (0,-0.5){};
\end{tikzpicture}
\hspace{8mm}
\begin{tikzpicture}[scale=0.6]
  \tikzstyle{every node}=[draw,thick,circle,fill=white,minimum size=6pt, inner sep=1pt]
    \node (r) at (0,0) [fill=black,minimum size=0pt,label=right:$r$] {};
    \node (1) at (0,1)  {$b$};
    \node (2a) at (-1,2) {$c$};
    \node (2b) at (1,2) [fill=black,minimum size=0pt,label=above:$3$]{};
    \node (3a) at (-2,3)  [fill=black,minimum size=0pt,label=above:$1$]{};
    \node (3b) at (-0,3)  [fill=black,minimum size=0pt,label=above:$2$]{};
    \draw  (r) -- (1)--(2a)--(3a);
    \draw         (1)--(2b);
    \draw              (2a)--(3b);
\node[minimum size=0pt,inner sep=0pt,label=below:$T_3$] (name) at (0,-0.5){};
\end{tikzpicture}
\hspace{8mm}
\begin{tikzpicture}[scale=0.6]
  \tikzstyle{every node}=[draw,thick,circle,fill=white,minimum size=6pt, inner sep=1pt]
    \node (r) at (0,0) [fill=black,minimum size=0pt,label=right:$r$] {};
    \node (1) at (0,1)  {$b$};
    \node (2a) at (-1,2) {$c$};
    \node (2b) at (1,2) {$b$};
    \node (3a) at (-1.5,3)  [fill=black,minimum size=0pt,label=above:$1$]{};
    \node (3b) at (-0.5,3)  [fill=black,minimum size=0pt,label=above:$2$]{};
    \node (3c) at (0.5,3)  [fill=black,minimum size=0pt,label=above:$3$]{};
    \node (3d) at (1.5,3)  [fill=black,minimum size=0pt,label=above:$4$]{};
    \draw  (r) -- (1)--(2a)--(3a);
    \draw         (1)--(2b)--(3c);
    \draw              (2a)--(3b);
    \draw              (2b)--(3d);
\node[minimum size=0pt,inner sep=0pt,label=below:$T_4$] (name) at (0,-0.5){};
\end{tikzpicture}
\caption{Tree monomials}\label{xxfig3-xxex1.10}
\end{figure}

\begin{definition}
\label{xxdef1.11}
Suppose $T=(\tau,\x)\in\TM$ and $\tau'$ is a subtree of
$\tau$. Then the tree monomial $T':=(\tau',\x|_{\Int(\tau')})$
is called a \emph{submonomial} of $T$. The tree monomial $T$
is said \emph{divisible} by $T_1=(\tau_1,\x_1)\in\TM$ if
$\tau$ contains a subtree $\tau_1'$ isomorphic to $\tau_1$
via $\phi:\tau_1'\to \tau_1$ and $\x(v)=\x_1(\phi(v))$ for
all $v\in\Int(\tau_1')$.
\end{definition}

A \emph{tree polynomial} in $\XX$ with coefficients in $\FF$
is an $\FF$-linear combination of tree monomials of the
same arity. The \emph{support} of a tree polynomial $f$,
denoted by $\Supp(f)$, is the set of all tree monomials
that appear in $f$ with nonzero coefficients. The arity
of $f$ is defined to be the arity of a tree monomial in
$\Supp(f)$. Denote the vector space of all tree polynomials
of arity $n\geq 1$ by $\TX(n)$ and $\TX(0)=0$. (Note that we
only consider $\XX$ with $\XX(0)=\emptyset$ in this paper.) Let
$\TX:=\{\TX(n)\}_{n\geq0}$.
In order to make $\TX$ an operad, we need to define
compositions of tree polynomials. We first define the
graftings of PRTs.

\begin{definition} \cite[Definition 3.3.3.2]{BD16}
\label{xxdef1.12}
Suppose that $\tau_1$ and $\tau_2$ are PRTs.
Let $l\in\Leaves(\tau_1)=\{1,2,\dots,\Ar(\tau_1)\}$.
We define a PRT $\tau_1\circ_l\tau_2$, called the result of
\emph{partial grafting} of $\tau_2$ to $\tau_1$ at $l$, as follows:
\begin{align*}
\Root(\tau_1\circ_l\tau_2)&:=\Root(\tau_1),\\
\Int(\tau_1\circ_l\tau_2)&:=\Int(\tau_1)\sqcup\Int(\tau_2),\\
\Leaves(\tau_1\circ_l\tau_2)&
:=\Leaves(\tau_1)\sqcup\Leaves(\tau_2)\setminus\{l\};
\end{align*}
the parent function and the planar structure on
$\tau_1\circ_l\tau_2$ are induced respectively
by the respective parent functions and planar structures
of $\tau_1$ and $\tau_2$ with two exceptions: for the
only vertex $v\in\Child_{\tau_2}(\Root(\tau_2))$, define
$\Parent_{\tau_1\circ_l\tau_2}(v):=\Parent_{\tau_1}(l)$;
the total order needed by the planar structure puts $v$
in the place of $l$.
\end{definition}

Partial graftings of PRTs induce partial compositions of
tree monomials.

\begin{definition}
\label{xxdef1.13}
Given two tree monomials $T_1=(\tau_1,\x_1)$ and
$T_2=(\tau_2,\x_2)$, we define the \emph{partial
composition} $T_1\circ_lT_2$ for $1\leq l\leq \Ar(T_1)$
to be the tree monomial $T=(\tau,\x)$,
where $\tau=\tau_1\circ_l\tau_2$ and
\[
\x(v):=
\begin{cases}
    \x_1(v)&v\in\Int(\tau_1),\\
    \x_2(v)&v\in\Int(\tau_2).
\end{cases}
\]
\end{definition}

\begin{example}
\label{xxex1.14}
Let $T_1$ and $T_2$ be the same as in Example \ref{xxex1.10},
see Figure \ref{xxfig3-xxex1.10}. All possible compositions
$T_1\circ_l T_2$ are demonstrated in Figure
\ref{xxfig4-xxex1.14}.
\end{example}

\begin{figure}[h]
\centering
\begin{tikzpicture}[scale=0.6]
  \tikzstyle{every node}=[draw,thick,circle,fill=white,minimum size=6pt, inner sep=1pt]
    \node (r) at (0,0) [fill=black,minimum size=0pt,label=right:$r$] {};
    \node (1) at (0,1)  {$a$};
    \node (2) at (0,2) {$b$};
    \node (3a) at (-1,3)  {$b$};
    \node (3aa) at (-2,4) [fill=black,minimum size=0pt,label=above:$1$]{};
    \node (3ab) at (0,4) {$c$};
    \node (3aba) at (-1,5) [fill=black,minimum size=0pt,label=above:$2$]{};
    \node (3abb) at (1,5) [fill=black,minimum size=0pt,label=above:$3$]{};
    \node (3b) at (1,3)  [fill=black,minimum size=0pt,label=above:$4$]{};
    \draw  (r) -- (1)--(2)--(3a)--(3aa);
    \draw                   (3a)--(3ab)--(3aba);
    \draw                         (3ab)--(3abb);
    \draw              (2)--(3b);
\node[minimum size=0pt,inner sep=0pt,label=below:$T_1\circ_1 T_2$] (name) at (0,-0.5){};
\end{tikzpicture}
\hspace{15mm}
\begin{tikzpicture}[scale=0.6]
  \tikzstyle{every node}=[draw,thick,circle,fill=white,minimum size=6pt, inner sep=1pt]
    \node (r) at (0,0) [fill=black,minimum size=0pt,label=right:$r$] {};
    \node (1) at (0,1)  {$a$};
    \node (2) at (0,2) {$b$};
    \node (3a) at (-1,3)  [fill=black,minimum size=0pt,label=above:$1$]{};
    \node (3b) at (1,3)  {$b$};
    \node (3ba) at (0,4)  [fill=black,minimum size=0pt,label=above:$2$]{};
    \node (3bb) at (2,4) {$c$};
    \node (3bba) at (1,5) [fill=black,minimum size=0pt,label=above:$3$]{};
    \node (3bbb) at (3,5) [fill=black,minimum size=0pt,label=above:$4$]{};
    \draw  (r) -- (1)--(2)--(3a);
    \draw              (2)--(3b)--(3ba);
    \draw                   (3b)--(3bb)--(3bba);
    \draw                         (3bb)--(3bbb);
    \draw              (2)--(3b)--(3ba);
\node[minimum size=0pt,inner sep=0pt,label=below:$T_1\circ_2 T_2$] (name) at (0,-0.5){};
\end{tikzpicture}
\caption{Partial compositions of tree monomials}\label{xxfig4-xxex1.14}
\end{figure}
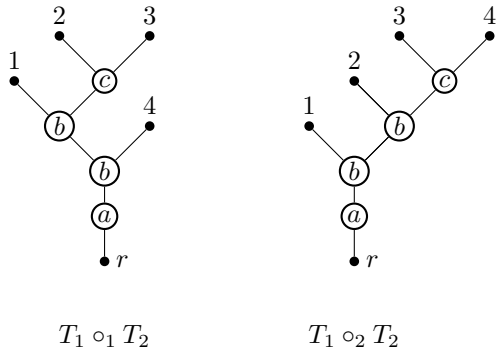

Extending compositions of tree monomials by multilinearity
to the collection $\TX$ gives \emph{partial compositions}
of tree polynomials:
\[
\circ_i:\TX(n)\otimes \TX(m)\to \TX(n+m-1),
\alpha\otimes \beta\mapsto\alpha\circ_i\beta,\
1\leq i\leq n.
\]

The following lemma is easy to prove.

\begin{lemma}
\label{xxlem1.15}
Equipped with the partial compositions defined above,
$\TX$ is the \emph{free reduced ns operad generated by $\XX$}.
\end{lemma}

An \emph{ideal} $\II$ of an operad $\PP$ is a subcollection
of $\PP$ such that each composition $f\circ_i g$ belongs to
$\II$ whenever $f$ or $g$ belongs to $\II$. Suppose
$\mathcal{S}$ is a subcollection of $\PP$. The
\emph{ideal of $\PP$ generated by $\mathcal{S}$}, denoted
by $(\mathcal{S})$, is the smallest (by inclusion) ideal of
$\PP$ containing $\mathcal{S}$.

We are ready now to define a presentation of an operad by
generators and relations.

\begin{definition}
\label{xxdef1.16}
Suppose that an operad $\PP$ is the quotient of the free
operad $\TX$ by some ideal $\II$, and that $\II$ is generated
by a subcollection $\mathcal{R}\subset \II$. We say that the
operad $\PP$ is \emph{presented by generators $\XX$ and
relations $\mathcal{R}$}. We call $\PP$ \emph{finitely generated}
(resp. \emph{finitely presented})
if $\PP$ can be presented by a finite set $\XX$ of generators,
i.e., $\sqcup_{n\geq0} \XX(n)$ is a finite set
(resp. by a finite set $\XX$ of generators and a finite
dimensional subcollection $\mathcal{R}$ of relations).
\end{definition}

\section{\gsbs\ of ns operads}
\label{xxsec2}
In this section, we follow the ideas in \cite[Chapter 3]{BD16}
to introduce \gsb\ (also known as \gb) theory for ns operads.
Similarly to the case of associative algebras, the \gsb\
method is useful for the computation of GK-dimension of an
operad.

Let $\XX$ be an operation alphabet and $\XX^*$ be the free
monoid generated by $\XX$. Recall that a total order $>$ on
$\XX^*$ is called a \emph{monomial order} on $\XX^*$ if $>$
is a well-order and $u_{1}>u_{2}$ implies
$u_{1}u_{3}>u_{2}u_{3}$ and $u_{3}u_{1}>u_{3}u_{2}$ for all
$u_{1},u_{2},u_{3}\in\XX^*$.

A collection of total orders $\succ_n$ of $\TM(n)$, $n\geq0$,
is called a \emph{monomial order} on $\TM$ if the following
conditions are satisfied:
\begin{enumerate}[(i)]
\item
each $\succ_n$ is a well-order;
\item
each partial composition is a strictly increasing function
in each of its arguments, i.e., if $T_0,T_0'\in \TM(m)$,
$T_1,T_1'\in \TM(n)$, $1\leq i\leq m$, then
\begin{align*}
T_0\circ_i T_1\succ_{m+n-1} T_0'\circ_i T_1&\text{ if }
    T_0\succ_m T_0',\\
T_0\circ_i T_1\succ_{m+n-1} T_0\circ_i T_1'&\text{ if }
    T_1\succ_n T_1'.
\end{align*}
\end{enumerate}

Now we introduce a monomial order on $\TM$. Let $T=(\tau, x)$
be a tree monomial. For each leaf $l$ of $\tau$, we record
the labels of internal vertices of the path from the root to
$l$, forming a word in alphabet $\XX$. The sequence of these
words, ordered by the planar structure on $\Leaves(\tau)$, is
called the \emph{path sequence} of the tree monomial $T$,
denoted by $\Path(T)$.

\begin{example}
\label{xxex2.1}
The path sequences of the tree monomials in Figure
\ref{xxfig3-xxex1.10} are
\begin{gather*}
    \Path(T_1)=(ab,ab), \quad
    \Path(T_2)=(b, bc, bc),\\
    \Path(T_3)=(bc,bc,b), \quad
    \Path(T_4)=(bc,bc,bb,bb).
\end{gather*}
\end{example}

It is not difficult to prove that a tree monomial is uniquely
determined by its path sequence, see \cite[Lemma 3.4.1.4]{BD16}.
Given a monomial order $>$ on $\XX^*$, we define an order
(still denoted by $>$, called the \emph{path extension} of
the monomial order on $\XX^*$) on $\TM$ by using path sequences
of tree monomials. Suppose $T,T'\in\TM$, $\Path(T)=(u_1,u_2,\dots,u_m)$
and $\Path(T')=(u'_1,u'_2,\dots,u'_n)$. We say $T>T'$ if either
$m>n$, or $m=n$ and there exists $1\leq i\leq m$ such that
$u_1=u'_1,u_2=u_2',\dots,u_{i-1}=u'_{i-1}, u_i>u_i'$.

The following lemma is copied from \cite[Proposition 3.4.1.6]{BD16}.

\begin{lemma}
\label{xxlem2.2}
Suppose $\XX=\sqcup_{n\geq 1} \XX(n)$, namely, $\XX(0)=\emptyset$.
The path extension of a monomial order on $\XX^*$ is a monomial
order on $\TM$.
\end{lemma}

Note that the above statement does not hold if we allow internal
vertices of arity zero in a PRT. See the following example.

\begin{example}
\label{xxex2.3}
Let $\XX=\{a,b\}$ with $a\in\XX(2)$ and $b\in\XX(0)$.
\begin{figure}[h]
		\centering
		\begin{tikzpicture}[scale=0.6]
			\tikzstyle{every node}=[draw,thick,circle,fill=white,minimum size=6pt, inner sep=1pt]
			\node (r) at (0,0) [fill=black,minimum size=0pt,label=right:$r$] {};
			\node (2) at (0,1) {$a$};
			\node (3a) at (-1,2)  {$b$};
			\node (3b) at (1,2)  [fill=black,minimum size=0pt,label=above:$1$]{};
			\draw  (r) --(2)--(3a);
			\draw              (2)--(3b);
			\node[minimum size=0pt,inner sep=0pt,label=below:$T_1$] (name) at (0,-0.5){};
		\end{tikzpicture}
		\hspace{15mm}
		\begin{tikzpicture}[scale=0.6]
			\tikzstyle{every node}=[draw,thick,circle,fill=white,minimum size=6pt, inner sep=1pt]
			\node (r) at (0,0) [fill=black,minimum size=0pt,label=right:$r$] {};
			\node (2) at (0,1) {$a$};
			\node (3a) at (-1,2)  [fill=black,minimum size=0pt,label=above:$1$]{};
			\node (3b) at (1,2)  {$b$};
			\draw  (r) --(2)--(3a);
			\draw              (2)--(3b);
			\node[minimum size=0pt,inner sep=0pt,label=below:$T_2$] (name) at (0,-0.5){};
		\end{tikzpicture}
	\caption{Tree monomials with $\XX(0)\neq \emptyset$}\label{xxfig5-xxex2.3}
	\end{figure}
Consider the tree monomials $T_1$ and $T_2$ in Figure \ref{xxfig5-xxex2.3}.
The path sequences of the tree monomials are
	\begin{gather*}
		\Path(T_1)=(ab,a), \quad
		\Path(T_2)=(a,ab ).
	\end{gather*}
Under the path extension of a degree-lexicographic order on $\XX^*$, we
have $T_1>T_2$. However, if we denote $T_3$ the tree monomial with only
one internal vertex labeled by $b$, then the compositions
$T_1\circ_1 T_3$ and $T_2\circ_1 T_3$ are equal, which shows that the
path extension is not a monomial order.
\end{example}

Fix a monomial order on $\TM$. Suppose
\[
g=a_1U_1+a_2U_2+\dots+a_nU_n\in \TX
\]
where $n\geq 1$, each $a_i\in \FF^*$, $U_i\in \TM$ and
$U_1>U_2>\cdots> U_n$. Then $U_1$ (respectively, $a_1$, $a_1U_1$)
is called the \emph{leading monomial} (respectively,
\emph{leading coefficient}, \emph{leading term}) of $g$,
denoted by $\lm{g}$ (respectively, $\lc(g)$, $\lt(g)$).
We say $g$ is \emph{monic} if $\lc(g)=1$.

Suppose $\GG\subseteq \TX$. Let $\lm{\GG}:=\{\lm g: g\in \GG\}$
and
\[
\Irr(\GG):=\{u\in \TM: u \text{ is not divisible by }\lm{g},\
\forall g\in \GG\}.
\]
A tree polynomial $f$ is \emph{reduced} with respect to
$\GG$ if $\Supp(f)\subseteq \Irr(\GG)$. We say $\GG$ is
\emph{self-reduced} if, for all $g\in \GG$, $g$ is monic
and reduced with respect to $\GG\setminus\{g\}$.

Note that, for an ideal $\II$ of $\TX$, the vector space
${\FF}\lm{\II}$ spanned by the leading monomials $\lm{\II}$
is also an ideal (see \cite[Proposition 3.4.3.1]{BD16}),
which is exactly the ideal generated by $\lm{\II}$,
i.e., $(\lm{\II})={\FF}\lm{\II}$.

\begin{definition}
\label{xxdef2.4}
Let $\II$ be an ideal of $\TX$. A subcollection $\GG$ of
$\II$ is called a \emph{\gsb} for $\II$ (or for the quotient
operad $\TX/\II$) if the ideal generated by $\lm{\GG}$
coincides with that generated by $\lm{\II}$, i.e.,
$(\lm{\GG})=(\lm{\II})$.
\end{definition}

It is clear that an ideal $\II\unlhd \TX$ is a \gsb\ for
$\II$.

\begin{proposition} \cite[Proposition 3.4.3.4]{BD16}
\label{xxpro2.5}
Let $\II$ be an ideal of $\TX$ and $\GG\subseteq \II$.
Then $\GG$ is a {\gsb} for $\II$  if and only if
$\Irr(\GG)$ forms an $\FF$-basis for the quotient $\TX/\II$.
\end{proposition}

\section{Gelfand-Kirillov dimension}
\label{xxsec3}
We refer to \cite{KL00} for basics about the Gelfand-Kirillov
dimension of associative algebras. The Gelfand-Kirillov
dimension of a locally finite operad is defined in
\cite[p.400]{KP15} and \cite[Definition 4.1]{BYZ20}.

\begin{definition} \cite[p.400]{KP15} \cite[Definition 4.1]{BYZ20}
\label{xxdef3.1}
Let $\PP$ be a locally finite operad. The
\emph{Gelfand-Kirillov dimension}, or \emph{GK-dimension}
for short, of $\PP$ is defined to be
\[
\GK(\PP):=\limsup_{n\to\infty}\log_n\left(\sum_{i=0}^n\dim \PP(i)\right)
\]
where $\dim$ stands for $\dim_{\FF}$.
\end{definition}

When we talk about the GK-dimension of an operad $\PP$, we
always implicitly assume that $\PP$ is locally finite. By
Lemma \ref{xxlem3.2} below, we might only consider the
GK-dimension of reduced (or reduced connected) operads from
now on.

Given two subcollections $\VV$ and $\WW$ of $\PP$, let
$\VV \circ \WW$ be the subcollection of $\PP$ spanned by
all elements of the form $v \circ_i w$ for $v\in \VV$,
$w\in \WW$ and $1\leq i\leq \Ar(v)$. Given a subcollection
$\VV$ of $\PP$, let $\VV^0=(0,\FF,0,0,\dots)$, and
inductively, let $\VV^m=\VV^{m-1}\circ \VV$ for $m\geq1$.
It is clear that $\VV^m=\left\{\VV^m(n)\right\}_{n\geq 0}$
where $\VV^m(n)$ is the subspace of $\PP(n)$ spanned by all
elements of arity $n$ that have the following form
\begin{align}
\label{E3.1.1}\tag{E3.1.1}
((\cdots((a_{1}\circ_{j_1}a_{2})\circ_{j_2}a_{3})
\circ_{j_3}\cdots )\circ_{j_{m-1}}a_m),
\text{ each } a_{i}\in \VV.
\end{align}
We call $\VV$ a \emph{generating subcollection} of $\PP$ if
\[
\PP=\sum_{m\geq0}\VV^m :=\left\{\sum_{m\geq0}\VV^m(n)\right\}_{n\geq0}.
\]

We say $\PP$ is a {\it finitely generated} operad with a
finite generating alphabet $\XX$, if $\PP$ has a finite
dimensional generating subcollection $\VV=\{\VV(n)\}_{n\geq0}$
where $\VV(n)$ is the space spanned by $\XX(n)$.

The following lemma is easy to prove and its proof is omitted.

\begin{lemma}
\label{xxlem3.2}
Let $\PP$ be a finitely generated locally finite operad.
\begin{enumerate}
\item[(1)]
We define a reduced operad associated with $\PP$:
\[
\PP_{r}:=(0,\PP(1),\PP(2),\PP(3),\dots).
\]
Then $\PP_r$ is finitely generated and locally finite.
\item[(2)]
We define a reduced connected operad
\[
\PP_{rc}:=(0,\FF,\PP(2),\PP(3),\dots).
\]
Then $\PP_{rc}$ is finitely generated and locally finite.
\item[(3)]
$\GK(\PP_{r})=\GK(\PP_{rc})=\GK(\PP)$.
\end{enumerate}
\end{lemma}

The following lemma gives a characterization of the
GK-dimension of a finitely generated operad.

\begin{lemma}
\label{xxlem3.3}
Suppose $\PP$ is a locally finite operad generated
by a finite dimensional subcollection $\VV$.
Then
\[
\GK(\PP)=\limsup_{n\to\infty}\log_n
\left(\dim\left(\sum_{i=0}^n \VV^i\right)\right).
\]
\end{lemma}

\begin{proof}
Let $t$ be the maximal arity of nonzero elements in $\VV$
and $s$ be an integer such that
$$(\PP(0),\PP(1),\dots,\PP(t),0,0,\dots)
\subseteq \sum_{j=0}^s \VV^j.$$
By the definition of $t$, we have
$\VV\subseteq (\PP(0),\PP(1),\dots,\PP(t),0,0,\dots)$.
Then, for every $n>0$,
\begin{align}
\label{E3.3.1}\tag{E3.3.1}
\sum^n_{i=0}\VV^i\subseteq
(\PP(0),\PP(1),\dots,\PP(nt-(n-1)),0,0,\dots)
\subseteq \sum^{ns}_{i=0}\VV^i.
\end{align}
Consequently,
\begin{align*}
\dim \left(\sum^n_{i=0}\VV^i\right)
\leq \dim\left(\sum^{nt-(n-1)}_{i=0}\PP(i)\right)
\leq \dim \left(\sum^{ns}_{i=0}\VV^i\right),\ n>0.
\end{align*}
Thus we have that
\[
\GK(\PP)=\limsup_{n\to\infty}\log_n\left(\dim
\left(\sum_{i=0}^n \PP(i)\right)\right)=
\limsup_{n\to\infty}\log_n\left(\dim\left(
\sum_{i=0}^n \VV^i\right)\right).
\]
\end{proof}

The following example shows that, in general, the
GK-dimension of an operad is not the supremum of
the GK-dimensions of its finitely generated suboperads.

\begin{example}
\label{xxex3.4}
Let $\alpha$ be any positive real number. Let $\PP$ be the operad
generated by infinitely many elements in $\XX$ such that
\[
\XX=\{\XX(n)\}_{n\geq 0},\ \XX(0)=\XX(1)=\emptyset,\
|\XX(n)|=\lfloor n^\alpha\rfloor-\lfloor (n-1)^\alpha\rfloor,
\ \forall n\geq 2,
\]
and subject to relations
\[
x_m\circ_i x_n=0, \quad \forall x_m\in \XX(m), x_n\in \XX(n),\quad
m,n\geq 2, \ 1\leq i\leq m.
\]
Then it is easy to check that $\GK(\PP)=\alpha>0$ but
$\GK(\PP')=0$ for all finitely generated suboperad $\PP'$ of $\PP$.
\end{example}

We will use the following nice construction which is related to
Warfield's example \cite{War84}.

\begin{example}
\label{xxex3.5}
We fix a real number $r$ strictly between $2$ and $3$ and
let $q$ be $\frac{r-1}{2}$ which is strictly between $\frac{1}{2}$
and $1$. Let $A$ be the quotient algebra $\FF\langle x_1,x_2\rangle/J$
generated by two elements $x_1,x_2$ of degree 1 and modulo the
monomial ideal $J$ generated by monomials having
degree $\geq 3$ in $x_2$ together with all monomials of the form
$$x_1^i x_2 x_1^j x_2 x_1^l$$
satisfying $j< n-\lfloor n^q \rfloor$ where $n=i+j+l+2$.

By an easy counting, $\dim A_0=1$, $\dim A_1=2$, and for each $n\geq 2$,
$$\dim A_n=1+n+\sum_{j=n-\lfloor n^q \rfloor}^{n-2}(n-1-j)=
1+n+\sum_{p=1}^{\lfloor n^q \rfloor-1}p.$$
As a consequence, we have
\begin{enumerate}
\item[(1)]
$\dim A_n$ is strictly increasing.
\item[(2)]
$$ \dim A_n= 1+n+\frac{1}{2}(\lfloor n^q \rfloor-1)(\lfloor n^q \rfloor)
\sim \frac{1}{2} n^{2q}=\frac{1}{2}n^{r-1}.$$
\item[(3)]
$\GK(A)=r$.
\end{enumerate}
\end{example}

Using the example above, we obtain the range
of possible values for the GK-dimension of an
associative algebra.

\begin{lemma}
\label{xxlem3.6}
\begin{enumerate}
\item[(1)]
For every integer $d\in \NN^\ast$, there is a graded algebra $A$
such that $\GK(A)=d$ and $\{\dim A_i\}_{i=0}^{\infty}$ is
weakly increasing.
\item[(2)]
There is a graded algebra $A$ such that $\GK(A)=\infty$.
\item[(3)]
For $d\in R_{\GK}\setminus \{0\}$,  there is
a graded algebra $A$ such that $\GK(A)=d$ and
$\{\dim A_i\}_{i=0}^{\infty}$ is strictly increasing.
\item[(4)]
All algebras in parts (1,2,3) can be taken to be monomial
algebras {\rm{(}}hence connected graded{\rm{)}}
finitely generated in degree 1.
\end{enumerate}
\end{lemma}

\begin{proof} (1)
We can take $A$ to be the commutative polynomial ring
$\FF[x_1,\cdots,x_d]$.

(2) We can take $A$ to be the free algebra $\FF\langle x_1,x_2\rangle$.

(3) If $d$ is an integer, it follows from part (1). If $d=\infty$,
it follows from part (2). Now we let $d$ be a finite non-integral
real number $>2$. Let $n$ be $\lfloor d\rfloor -2$ and $r=d-n$.
Then $n\geq 0$ and $2<r<3$. Let $A$ be the algebra given in Example
\ref{xxex3.5} and let $A'=A[x_1,\cdots,x_n]$. Then the assertion follows.

(4) It is well-known (see, e.g, \cite[Remark 4.1]{Bel15}) that
given a finitely generated associative algebra $A$, there is a
finitely generated monomial algebra $B=\FF\langle X\rangle/I$
such that $\GK(A)=\GK(B)$, where $\FF\langle X\rangle$ is the free
associative algebra generated by a finite set $X$ and
$I$ is an ideal consisting of words in $X$.
Since $A$ is connected graded, we even have that $A$ and $B$ have the
same Hilbert series.
\end{proof}

Below is a weak version of Theorem \ref{xxthm0.2}(2).

\begin{proposition}
\label{xxpro3.7}
For any $r\in R_{\GK}$, there
exists a finitely generated operad $\PP$ such that $\GK(\PP)=r$.
\end{proposition}

\begin{proof} By Lemma \ref{xxlem3.6}(3), there is a finitely
generated connected graded monomial algebra $A$ such that
$\GK(A)=r$. By Construction \ref{xxcon1.3}, there is a finitely
generated operad $\PP$ such that $\dim \PP(n)=\dim A_{n-1}$
for all $n\geq 1$. Therefore, by definition,
$\GK(\PP)=\GK(A)=r$. The assertion follows.
\end{proof}

\begin{proposition}
\label{xxpro3.8}
Suppose $\PP$ is a finitely generated and locally finite ns operad.
\begin{enumerate}
\item[(1)]
$\GK(\PP)=0$ if and only if $\PP$ is finite dimensional.
\item[(2)]
$\GK(\PP)$ cannot be strictly between 0 and 1.
\end{enumerate}
\end{proposition}

\begin{proof} (1)
The ``if'' part is clear.
For the ``only if'' part, suppose $\dim(\PP)=\infty$.
Since $\PP$ is finitely generated, there is a finite-dimensional
subcollection $(1_{\PP}\in) \VV$ that generates $\PP$. We claim that
$\VV^{m+1}\neq \VV^m$ for every $m$. Suppose to the
contrary that $\VV^{m+1}= \VV^m$ for some $m$. Then
by induction, one sees that $\VV^n=\VV^m$ for every
$n>m$. So $\PP=\cup_{n>m} \VV^n=\VV^m$, which is finite
dimensional. This yields a contradiction. Therefore
we proved the claim, and consequently,
$\dim \VV^m\geq m+1$ for every $m$. By Lemma \ref{xxlem3.3},
\[
\GK(\PP)=\limsup_{n\to \infty}\log_n\left(\sum_{i=0}^n\dim\VV^i\right)\geq
\lim_{n\to \infty}\log_n(n+1)=1,
\]
a contradiction.

(2) See the proof of part (1).
\end{proof}

The part (2) of the above proposition, as in the case of
associative algebras, there is a gap between $0$ and $1$ for
the GK-dimensions of finitely generated operads. This
is false if $\PP$ is infinitely generated, see
Example \ref{xxex3.4} and \cite[Corollary 6.12]{BYZ20}.

A \emph{monomial operad} means a quotient of free operad
by an ideal generated by tree monomials. The following lemma
implies that given an operad, there exists a {monomial operad}
with the same GK-dimension.

\begin{lemma}
\label{xxlem3.9}
Suppose $\II\unlhd \TX$. Then $\GK(\TX/\II)=\GK(\TX/(\lm \II))$.
\end{lemma}

\begin{proof}
Note that both $\II$ and $(\lm \II)$ are ideals and thus \gsbs.
Now the statement follows from Proposition \ref{xxpro2.5} and
$\Irr(\II)=\Irr((\lm \II))$.
\end{proof}

\section{Bergman's gap theorem}
\label{xxsec4}
\subsection{Single-branched ns operads}
In this subsection, we will introduce so-called
single-branched tree monomials and study their properties,
which will be used in the next subsection to prove Bergman's
gap theorem for operads.

Let $\tau$ be a PRT. An internal vertex $v$ is called a
\emph{top internal vertex} if its children are all leaves.
A \emph{branch} of $\tau$ is a path from the root to a top
internal vertex $v$, denoted by $\bran(v)$. A PRT is called a
\emph{single-branched tree} if it has exactly one branch.
The planar structure of $\tau$ induces an order on branches
of $\tau$:
\begin{center}
$\bran(v)>\bran(v')$ if $\h(v)>\h(v')$, or  $\h(v)=\h(v')$
and $v>v'$.
\end{center}
The maximal branch of $\tau$ is called the \emph{pivot branch}
of $\tau$. Denote by $\Piv(\tau)$ the set of all {vertices} of
$\tau$ that belong to the pivot branch of $\tau$. The top
internal vertex of $\tau$ belonging to $\Piv(\tau)$ is called
the \emph{pivot top internal vertex}, denoted by $\TIV(\tau)$.

A tree monomial is called \emph{single-branched} if its
underlying tree is single-branched. An operad is called
\emph{single-branched} if it has an $\FF$-basis that consists
of single-branched tree monomials.

A tree monomial $T$ is called \emph{right normal} if it can
be written in terms of partial compositions of generating
operations with parentheses from right to left,
i.e.,
\begin{equation}
\label{E4.0.1}\tag{E4.0.1}
T=(x_1\circ_{i_1}(\cdots (x_{n-2}\circ_{i_{n-2}}
(x_{n-1}\circ_{i_{n-1}}x_n))\cdots)).
\end{equation}
\emph{Left normal} tree monomials are defined similarly,
see \eqref{E3.1.1}. Note that every tree monomial is left
normal, but not necessarily right normal. For a right normal
tree monomial, we usually write without parentheses
(and/or composition symbols $\circ_i$ sometimes) for short
when no confusion arises,
e.g.,
\begin{equation}
\label{E4.0.2}\tag{E4.0.2}
T=x_1\circ_{i_1}x_{2}\circ_{i_2}\cdots \circ_{i_{n-1}}x_n
=x_1x_2\cdots x_n.
\end{equation}
The following lemma is straightforward.

\begin{lemma}
\label{xxlem4.1}
Let $T$ be a tree monomial. Then the following statements
are equivalent:
\begin{enumerate}[(i)]
\item
$T$ is single-branched;
\item
$T$ is right normal;
\item
$\wt(T)=\h(T)$.
\end{enumerate}
\end{lemma}

\begin{example}
\label{xxex4.2}
In Figure \ref{xxfig3-xxex1.10}, $T_1,T_2$ and $T_3$ are
single-branched and $T_4$ is not single-branched.
\end{example}

A single-branched tree monomial
$w=x_1\circ_{i_1}x_2\circ_{i_2}\cdots \circ_{i_{n-1}}x_n$
is called \emph{periodic} if there exists a positive
integer $p<n$ such that $x_j=x_{j+p}$ for all
$1\leq j\leq n-p$ and $i_{j'}=i_{j'+p}$ for all
$1\leq j'\leq n-p-1$. The integer $p$ is called a
\emph{local period} of $w$ and the smallest local period
of $w$ is called the \emph{minimal period} of $w$.
Given a periodic tree monomial
$w=x_1\circ_{i_1}x_2\circ_{i_2}\cdots \circ_{i_{n-1}}x_n$
with minimal period $p$, by using its minimal period,
we can extend $w$ to the left and/or to the right to get
a new single-branched tree monomial which contains $w$ as
a submonomial. More precisely, $w$ can be extended to
\[
w_{m,l}:=x_{-m}\circ_{i_{-m}}\cdots \circ_{i_{-1}} x_{0}
\circ_{i_{0}}x_1\circ_{i_1}x_2\circ_{i_2}\cdots
\circ_{i_{n-1}}x_n\circ_{i_n} \cdots \circ_{i_{l-1}}x_l
\]
where $m\geq -1$, $l\geq n$, for each integer $q$ between
$-m$ and $l$ (say, $-m\leq q=ps+r\leq l$, $s,r\in \ZZ$, $1\leq r\leq p$),
$x_q=x_r$ and $\circ_{i_q}=\circ_{i_r}$ (except $\circ_{i_l}$).
It is clear that $p$ is still the minimal period of $w_{m,l}$.
If a positive integer $p'$ is a local period of all extensions
$w_{m,l}$ of $w$, we call $p'$ a \emph{period} of $w$.
For example, given
\[
w=a\circ_1 a \circ_1 b \circ_1 a\circ_1 a \circ_1 b
\circ_1 a\circ_1 a,
\]
then $3$ is the minimal period of $w$;
$6$ is a period of $w$; $7$ is a local period (but not a
period) of $w$.

The following lemma on periods is easy to prove.

\begin{lemma}
\label{xxlem4.3}
Suppose $w$ is a single-branched tree monomial of minimal
period $p$. If $l$ is a period of $w$, then $p$ divides $l$.
\end{lemma}

\begin{proof}
Suppose to the contrary that $l=pq+r$ with $q,r\in \NN$ and
$0<r<p$. Assume $w=x_1\circ_{i_1}x_2\circ_{i_2}\cdots
\circ_{i_{n-1}}x_n$. Consider the extension
\[
w_{l,n+p}=x_{-l}\circ_{i_{-l}}\cdots \circ_{i_{-1}} x_{0}
\circ_{i_{0}}x_1\circ_{i_1}\cdots \circ_{i_{n-1}}x_n
\circ_{i_{n}}\cdots\circ_{i_{n+p-1}}x_{n+p}.
\]
For all $1\leq i\leq n-r$, since $l$ and $pq$ are local
periods of $w_{l,n+p}$,  we have that
$x_{i+r}=x_{i+l-pq}=x_{i-pq}=x_{i}$. Thus $r$ is a local
period of $w$, contradicting the minimality of $p$.
\end{proof}

The following is an analogue of \cite[Lemma 2.3]{KL00}.

\begin{lemma}
\label{xxlem4.4}
Let $w$ be a single-branched tree monomial of height $n>0$.
Suppose that $w$ is periodic with minimal period $p<n$ and
has two equal submonomials
\begin{equation}
\label{E4.4.1}\tag{E4.4.1}
x_{i+1}\circ_{\alpha_{i+1}}\cdots \circ_{\alpha_{i+r-1}}x_{i+r}=
x_{j+1}\circ_{\alpha_{j+1}}\cdots \circ_{\alpha_{j+r-1}}x_{j+r}
\end{equation}
of height $r\geq p$, $0\leq i<j$. Then $p$ divides $j-i$.
\end{lemma}

\begin{proof}
We modify the proof of \cite[Lemma 2.3]{KL00} to take care of
the composition indices. For the manipulations described below,
if the tree monomial $w$ is too short, consider it in its extensions.

Since $r\geq p$, $i+r\geq i+p$. Similarly, $j+r\geq j+p$.
Since $p$ is the minimal period of $w$,
$$x_i=x_{i+p}=x_{j+p}=x_j.$$
By periodicity, we have the following equality of two sets
$$\{\circ_{\alpha_i}, \circ_{\alpha_{i+1}},
\cdots, \circ_{\alpha_{i+p-1}}\}
=\{\circ_{\alpha_j}, \circ_{\alpha_{j+1}},
\cdots, \circ_{\alpha_{j+p-1}}\}.$$
By \eqref{E4.4.1}, we also have
$$\{\circ_{\alpha_{i+1}},
\cdots, \circ_{\alpha_{i+p-1}}\}
=\{\circ_{\alpha_{j+1}},
\cdots, \circ_{\alpha_{j+p-1}}\},$$
which forces that $\circ_{\alpha_i}=
\circ_{\alpha_j}$. Thus $w$ has two equal submonomials
\[
x_i\circ_{\alpha_i}x_{i+1}\circ_{\alpha_{i+1}}\cdots
\circ_{\alpha_{i+r-1}}x_{i+r}=
x_j\circ_{\alpha_j}x_{j+1}\circ_{\alpha_{j+1}}\cdots
\circ_{\alpha_{j+r-1}}x_{j+r},
\]
which have height $r+1$. Similarly we can extend the
above two submonomials to any height $\geq r$.

Next we show that $j-i$ is a period of $w$. In any
extension $w_{m,m'}$ of $w$ (for some $m\geq 0$ and
$m'\geq n$), let $-m\leq l\leq m'-(j-i)$ and $t\in\ZZ$
such that $l+tp=i+s$, $0\leq s\leq p-1$. Then
\[
x_{l+(j-i)}=  x_{l+(j-i)+tp}=x_{i+s+(j-i)}=x_{j+s}
=x_{i+s}=x_{l+tp}=x_l
\]
for $-m\leq l\leq m'-(j-i)$ and
\[
\alpha_{l+(j-i)}=  \alpha_{l+(j-i)+tp}
=\alpha_{i+s+(j-i)}=\alpha_{j+s}
=\alpha_{i+s}=\alpha_{l+tp}=\alpha_l
\]
for $-m\leq l\leq m'-(j-i)-1$. Hence $w$ has a period
$j-i$. Now it follows from Lemma \ref{xxlem4.3} that
$p$ divides $j-i$.
\end{proof}

We modify \cite[Lemma 2.4]{KL00} and obtain the following

\begin{lemma}
\label{xxlem4.5}
Suppose $\XX$ is a finite operation alphabet.
Let $W$ be a set of single-branched tree monomials on
$\XX$ such that all submonomials of elements in $W$ still
belong to $W$. Suppose that, for some positive integer $d
\geq 3$, $W$ contains at most $d-1$ tree monomials of
height $d$. Then $W$ contains at most $(d-1)^3$ tree
monomials of height $h$ for all $h\geq d$.
\end{lemma}

\begin{proof} Let $w$ be a single-branched tree monomial
of height $h$ of the form \eqref{E4.0.2} and let $W_h$
be the subset of $W$ consisting of $w\in W$ of height $h$.

If $h\leq 2d-1$, then $w$ is completely determined by
its top and bottom submonomials of height $d$. There are
at most $(d-1)^2$ possibilities for this case. So
$$|W_h|\leq (d-1)^2< (d-1)^3.$$

If $2d\leq h\leq 3d-2$, then $w$ is completely determined by
its top and bottom submonomials of height $d$ and the
submonomial $x_d\cdots x_{2d-1}$. There are
at most $(d-1)^3$ possibilities for this case. So
$$|W_h|\leq (d-1)^3.$$

For other cases, we need to use the following claim.

\textbf{Claim}: Suppose $w$ has height $\geq 2d-1$. Then
$w=w_1w_2w_3(=w_1\circ_{\beta_1} w_2\circ_{\beta_2}w_3)$
where $w_2$ is periodic of minimal period $p\leq d-1$,
$\h(w_2)\geq p+d$, $\h(w_1)\leq d-p$ and $\h(w_3)\leq d-p$.
The periodic submonomial
$w_2$ is $x_{j+1}\cdots x_{r}$ for some integers $j\leq d-p$
and $r\geq h-(p-d)$.

Note that our lemma for $h\geq 3d-1(> 2d-1)$ follows from the claim.
In fact, by the claim, $w=x_1\cdots x_h$ for $h\geq 3d-1$
is uniquely determined by its bottom and top submonomials of
height $d$ (each has at most $d-1$ possibilities by assumption),
and the periodic submonomial $w'_2:=x_{d-p+1}\cdots x_{h-(d-p)}$.
Note that $w'_2$ is a  submonomial of $w_2$ with minimal period
$p\leq d-1$. Hence $w'_2$ has at most $d-1$ possibilities,
determined by it top submonomial of height $d$.
As a result, there are at most $(d-1)^3$ possibilities for $w$
as desired.

\noindent
{\it Proof of the Claim:} We prove this claim by induction on $h$.

Initial step: Suppose $h=2d-1$. Then $w=x_1x_2\cdots x_{2d-1}$
has $d$ submonomials of height $d$, all belonging to $W$ by
assumption. Hence two of these tree monomials must be equal, say
\[
x_{j+1}\cdots x_{j+d}=x_{j+p+1}\cdots x_{j+p+d}
\]
where $j\geq 0$, $p\geq1$, $j+p\leq d-1$, $p$ chosen as small as
possible. Then $x_{j+1}\cdots x_{j+d}$ and
$x_{j+p+1}\cdots x_{j+p+d}$ have a nonempty overlap and thus
\[
w_2:=x_{j+1}\cdots  x_{j+p+d}
\]
is a periodic submonomial (of $w$) of minimal period $p\leq d-1$
and of height $\h(w_2)=p+d$. The bottom submonomial
$w_1:=x_1\cdots x_j$ and the top submonomial
$w_3:=x_{j+p+d+1}\cdots x_{2d-1}$ are of height $\leq d-p-1<d-p$.
Thus the claim holds for $h=2d-1$ and we finish the initial step.

Inductive step: Suppose the claim is true for an $h\geq 2d-1$.
Consider $w=x_1x_2\cdots x_{h+1}\in W$ of height $h+1$. By
hypothesis,
\[
w=x_1(x_2\cdots x_j)(x_{j+1}\cdots x_{j+r})
(x_{j+r+1}\cdots x_hx_{h+1})
\]
where $\h(x_2\cdots x_j)=j-1\leq d-p$,
$\h(x_{j+r+1}\cdots x_hx_{h+1})=h-j-r+1\leq d-p$,
and $x_{j+1}\cdots x_{j+r}$ is periodic of minimal period
$p\leq d-1$ and height $r\geq p+d$. If $j-1<d-p$, i.e.,
$j\leq d-p$, then the claim follows by taking
\[
w_1=x_1\cdots x_j,\quad
w_2=x_{j+1}\cdots x_{j+r},\quad
w_3=x_{j+r+1}\cdots x_hx_{h+1}.
\]
Now suppose $j-1=d-p$. We will show that the periodicity of
\[
x_{j+1}\cdots x_{j+r}
=x_{j+1}\circ_{\alpha_{j+1}}\cdots \circ_{\alpha_{j+r-1}}x_{j+r}
\]
can be extended down to include the term $x_j$, which completes
our inductive step. Note that
\[
j+r=d-p+1+r\geq d-p+1+d+p>2d.
\]
Thus $x_2\cdots x_{2d}$ is a proper submonomial of
$x_1\cdots x_{j+r}$ and contains two equal submonomials of
height $d$, say,
\[
x_{i+1}\cdots x_{i+d}=x_{i+n+1}\cdots x_{i+n+d},
\]
where $i\geq 1$, $n\geq 1$ and $i+n\leq d$. Since
$j=d-p+1\leq d$, both of these submonomials have an overlap
with the periodic tree monomial $x_{j+1}\cdots x_{j+r}$ and
the overlap contains at least their top (i.e., right) $d-(j-1)=p$
terms. By Lemma \ref{xxlem4.4}, we have that $p$ divides $n$
(say, $n=cp\geq p$ for some integer $c>0$) and hence
\[
i+d> d\geq j=d-p+1\geq d-n+1\geq i+1.
\]
Thus $x_j=x_{j+n}$ and ${\alpha_j}={\alpha_{j+n}}$. Note that
$j+1\leq j+p<j+r$ and $j+1\leq j+n< j+d <j+r$. By the
periodicity, we have that
\[
x_{j+p}=x_{j+n-(c-1)p}=x_{j+n}=x_j \text{ and }
\alpha_{j+p}=\alpha_{j+n-(c-1)p}=\alpha_{j+n}=\alpha_j,
\]
i.e., the periodicity of $x_{j+1}\cdots x_{j+r}$ extends down
to include $x_j$ as desired.

This finishes the inductive step, then the claim, and finally
the assertion in the lemma.
\end{proof}

The following example shows that the original
\cite[Lemma 2.4]{KL00} is not true for the setting of operads,
i.e., if we replace $d-1$ by $d$ in Lemma \ref{xxlem4.5}, then
the statement does not hold any further.

\begin{example}
\label{xxex4.6}
Let $W$ be the set of single-branched tree monomials of the form
\[
x\circ_{i_1}x\circ_{i_2}\cdots \circ_{i_{n-1}}x
\]
where $\Ar(x)=2$, $n\geq 1$, at most one composition index $i_j$
equals $2$ and the other indices are all $1$. It is easy to see
that all submonomials of elements in $W$ are still in $W$ and that,
for $n\geq 1$, $W$ contains exactly $n$ distinct monomials of height $n$.
\end{example}

\begin{lemma}
\label{xxlem4.7}
Let $\PP$ be a locally finite operad with $\GK(\PP)<2$.
Let $a$ and $b$ be two positive real numbers.
\begin{enumerate}
\item[(1)]
There are infinitely many integer $n$ such that
$\dim \PP(n) < a n-b$.
\item[(2)]
Suppose $\PP$ is finitely generated by $\VV$. Then
there are infinitely many integer $n$ such that
$\dim (\sum_{i=0}^n \VV^i)/(\sum_{i=0}^{n-1} \VV^i) < a n-b$.
\end{enumerate}
\end{lemma}

\begin{proof} (1)
Suppose to the contrary that $\dim \PP(n)\geq an-b$
for all $n\gg 0$. Then, there is an $m$, such that
$$\GK(\PP)=\limsup_{n\to \infty}
\log_n (\sum_{i=0}^n \dim \PP(i))
\geq \limsup_{n\to \infty}
\log_n (\sum_{i=m}^n (a i-b))=2,$$
yielding a contradiction.
The proof of part (2) is similar to the above.
\end{proof}

\subsection{Gap theorem for ns operads}
\label{xxsec4.2}
In this subsection, we first prove three lemmas and
then use them to prove an analogue of Bergman's gap
theorem for finitely generated ns operads.

Roughly speaking, the following lemma says that,
if the GK-dimension of a finitely generated operad is less
than $2$, then there exists a uniform upper bound for
the heights of maximal subtrees that are rooted at the
pivot branch and do not contain the pivot top internal vertex.

\begin{lemma}
\label{xxlem4.8}
Let $\XX$ be a finite operation alphabet and $\PP=\TX/\II$.
Write $\VV=\FF \XX$. Suppose $\dim
(\sum_{i=0}^d \VV^i)/(\sum_{i=0}^{d-1} \VV^i)\leq d-2$ for
some $d$ {\rm{(}}or $\GK(\PP)<2${\rm{)}}.
Then there exists a positive integer $M_1$ such that,
for all $\tau\in\Tr(\Irr(\II))$,
\[
\h(\Twig_{\tau}(v,v'))<M_1
\]
where $v\in \Piv(\tau)$, $v'\in \Int(\tau)\setminus\Piv(\tau)$.
\end{lemma}

\begin{proof}
Suppose to the contrary that for each $n\in\NN$ there is a tree
monomial $T_n\in \Irr(\II)$ with underlying PRT $\tau_n$ and a
maximal subtree $\tau_n':=\msub_{\tau_n}(v_n,v_n')$ such that
$h_n:=\h(\tau_n')\geq n$, where $v_n\in \Piv(\tau_n)$ and
$v'_n\in \Int(\tau_n)\setminus\Piv(\tau_n)$. Assume that the
vertices in $\Piv(\tau_n')$ are labelled from $v_n$ to
$\TIV(\tau_n')$ by generating operations $x,y_1,y_2,\dots,y_{h_n}$,
and that the vertices in $\Int(\tau_n)\cap\Piv(\tau_n)$ with height
$\geq \h(v_n)$ are labelled from $v_n$ to $\TIV(\tau_n)$ by
$x,x_1,x_2,\dots,x_m$ (see Figure \ref{xxfig5-xxlem4.8}, where
only a subset of $\Vert(\tau_n)$ are drawn and the pivot branch
is drawn vertically).
\begin{figure}[h]
\centering
\begin{tikzpicture}[scale=0.6]
  \tikzstyle{every node}=[draw,thick,circle,fill=white,minimum size=6pt, inner sep=1pt]
    \node (r) at (0,0) [fill=black,minimum size=0pt,label=right:$r$] {};
    \node (1) at (0,1)  {};
    \node (2) at (0,2.5) [label=right:$v_n$]{$\,~x~\, $};
    \node (3a) at (0,4.5)  {$\,x_1\,$};
    \node (3b) at (-1.5,3.5)  {$~y_1\,$};
    \node (4a) at (0,6.5)  {$x_m$};
    \node (4b) at (-3,5)  {$y_{_{h_n}}$};
    \draw  (r) -- (1);
    \draw [dotted,line width=0.6pt] (1)--(2);
    \draw (2)--(3a);
    \draw              (2)--(3b);
    \draw [dotted,line width=0.6pt] (3a)--(4a);
    \draw [dotted,line width=0.6pt] (3b)--(4b);
\node[minimum size=0pt,inner sep=0pt,label=below:$T_n$] (name) at (0,-0.5){};
\end{tikzpicture}
\caption{Tree monomial $T_n$}
\label{xxfig5-xxlem4.8}
\end{figure}
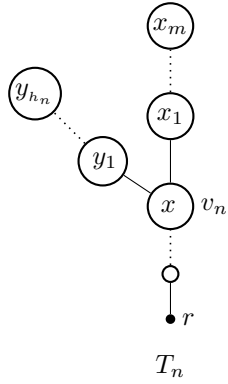

It follows from the definition of pivot branch that
$m\geq h_n\geq n$. Consider the submonomial $T''_n$ of $T_n$
with underlying PRT $\tau''_n\subseteq \tau_n$ for which
$\Root(\tau''_n)=\Parent_{\tau_n}(v_n)$ and
\[
\Int(\tau''_n)=\Piv(\tau_n')\cup \{v\in\Piv(\tau_n)
\cap\Int(\tau_n):\h(v)\geq \h(v_n)\}.
\]
Without loss of generality, we assume $T''_n$ is of the
form in Figure \ref{xxfig5-xxlem4.8} and denote it by
\[
T''_n=(x\circ_j(x_1x_2\cdots x_{m}))\circ_i(y_1y_2\cdots y_{h_n})
\]
where $i$ and $j$ are proper composition indices and $i<j$.
The following are distinct submonomials of $T_n''$ of weight $n$:
\[
(x\circ_{j}(x_1\cdots x_s))\circ_i(y_1\cdots y_{n-s-1}), \ 1\leq s\leq n-1.
\]
These submonomials are pairwise distinct since their
underlying PRTs are pairwise distinct. Thus $\Irr(\II)$ contains
at least $n-1$ elements of weight $n$, namely,
$\dim (\sum_{i=0}^n \VV^i)/(\sum_{i=0}^{n-1} \VV^i)\geq n-1$
for all $n$, yielding a contradiction.
\end{proof}

The following lemma says that, under suitable conditions,
if an internal vertex $v$ in the pivot branch has a child
that is an internal vertex and not in the pivot branch,
then $v$ must be either ``close'' to the root or ``close''
to the top of the PRT.

\begin{lemma}
\label{xxlem4.9}
Let $\XX$ be a finite operation alphabet and $\PP=\TX/\II$.
Write $\VV=\FF \XX$. Assume $\dim
(\sum_{i=0}^d \VV^i)/(\sum_{i=0}^{d-1} \VV^i)\leq d-3$ for
some $d\geq 3$ {\rm{(}}or $\GK(\PP)<2${\rm{)}}. Suppose that
\[
W:=\{T\in\Irr(\II): T \text{ not single-branched }\}\not=\emptyset.
\]
Then there exists a positive integer $M_2$ such that,
for all $\tau\in \Tr(W)$ and all $v\in\Piv(\tau)$ such that
there is a maximal subtree $\Twig(v,v')$ for some
$v'\in\Int(\tau)\setminus\Piv(\tau)$, the following inequality
holds
\[
\min\{\h(v),\h(\tau)-\h(v)\}<M_2.
\]
\end{lemma}

\begin{proof}
Suppose to the contrary that for each $n\in\NN^*$ there exist
$T_n\in W$ with $\Tr(T_n)=\tau_n$, $v_n\in\Piv(\tau_n)$, and
$v'_n\in\Int(\tau_n)\setminus\Piv(\tau_n)$ such that $\tau_n$
has a maximal subtree $\Twig(v_n,v_n')$ and
$\min\{\h(v_n),\h(\tau_n)-\h(v_n)\}\geq n$. Without loss of
generality, we suppose that $\Parent(v_n')=v_n$. Consider the
subtree $\tau_n'\subseteq \tau_n$ that has internal vertices
\[
\Int(\tau'_n)=\left(\Int(\tau_n)\cap\Piv(\tau_n)\right)\cup\{v_n'\}.
\]
Assume that the vertex $v_n'$ is labelled by generating
operation $x'$ and that the internal vertices in $\Piv(\tau_n)$
are labelled from bottom to top by generating operations
\[
x_m,x_{m-1},\dots,x_1,x,y_1,y_2,\dots,y_q
\]
for some integers $m\geq n-1$, $ q\geq n$
(see Figure \ref{xxfig6-xxlem4.9}, where only a subset of
$\Vert(\tau'_n)$ are drawn).
\begin{figure}[h]
\centering
\begin{tikzpicture}[scale=0.5]
  \tikzstyle{every node}=[draw,thick,circle,fill=white,minimum size=6pt, inner sep=1pt]
    \node (r) at (0,0) [fill=black,minimum size=0pt,label=right:$r$] {};
    \node (1) at (0,1.5)  {$x_m$};
    \node (2) at (0,3.5){$\,x_1\,$};
    \node (3a) at (0,5.5)  [label=right:$v_n$]{$~\, x~\, $};
    \node (3b) at (-2,7)  [label=left:$v_n'$]{$~x'\,$};
    \node (4a) at (0,7.5)  {$~y_1\,$};
    \node (5) at (0,9.5)  {$~y_q\,$};
    \draw  (r) -- (1);
    \draw [dotted,line width=0.6pt] (1)--(2);
    \draw (2)--(3a)--(4a);
    \draw              (3a)--(3b);
    \draw [dotted,line width=0.6pt] (4a)--(5);

\node[minimum size=0pt,inner sep=0pt,label=below:$T_n'$] (name) at (0,-0.5){};
\end{tikzpicture}
\caption{Submonomial $T'_n$ of $T_n$}
\label{xxfig6-xxlem4.9}
\end{figure}
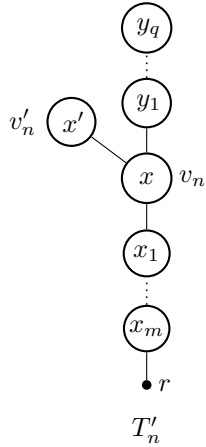

Consider following submonomial of $T_n'$ of weight $n+1$
\[
(x_sx_{s-1}\cdots x_1xy_1y_2\cdots y_{n-s-1})\circ_{i_s}x', \
1\leq s\leq n-1
\]
with suitable composition indices $i_{s}$. These submonomials are
pairwise distinct since so are their underlying PRTs. Thus
$\Irr(\II)$ contains at least $n-1$ elements of weight $n+1$,
namely, $\dim
(\sum_{i=0}^{n+1} \VV^i)/(\sum_{i=0}^{n} \VV^i)\geq n-1$, yielding
a contradiction.
\end{proof}

The following lemma is a generalization of Lemma \ref{xxlem4.5}.

\begin{lemma}
\label{xxlem4.10}
Retain the hypotheses of Lemma \ref{xxlem4.9}. Then there exist
positive integers $d_1$ and $d_2$ such that $\Irr(I)$ contains
at most $d_1$ tree monomials of weight $h$ for all $h\geq d_2$.
\end{lemma}

\begin{proof}
Let $M_1$ and $M_2$ be the same as in Lemmas \ref{xxlem4.8} and
\ref{xxlem4.9}. Suppose $T\in \Irr(I)$ and denote $\tau:=\Tr(T)$
and $n:=\wt(\tau)$. Assume that  $\h(T)$ is large enough for
the following decomposition (say, $\h(\tau)>2(M_1+2M_2)$). Let
$T_1$, $T_2$ and $T_3$ be the submonomials of $T$ whose underlying
PRTs $\tau_1$, $\tau_2$ and $\tau_3$ have internal vertices
\begin{align*}
\Int(\tau_1)&=\{v\in\Int(\tau):\h(v)\leq M_1+M_2\},\\
\Int(\tau_2)&=\{v\in\Int(\tau):M_1+M_2\leq \h(v)\leq \h(\tau)-M_2\},\\
\Int(\tau_3)&=\{v\in\Int(\tau):\h(v)\geq \h(\tau)-M_2\}.
\end{align*}
These subtrees are well-defined. In fact, by Lemmas \ref{xxlem4.8}
and \ref{xxlem4.9}, within each $\Int(\tau_i)$ vertices are
connected by the parent function of $\tau$ and $\tau$ has a
unique internal vertex $v_1$ of height $M_1+M_2$ (respectively,
$v_2$ of height $\h(\tau)-M_2$), whose parent is the root of
$\tau_2$ (respectively, $\tau_3$). Note that, by Lemmas
\ref{xxlem4.8} and \ref{xxlem4.9}, all
$v\in\Int(\tau)\setminus\Piv(\tau)$ are contained in
$\Int(\tau_1)\cup\Int(\tau_3)$ and thus $\tau_2$ is
single-branched. See Figure \ref{xxfig7-xxlem4.10} for an example
of such decomposition, where only a subset of the internal
vertices of $\tau$ are drawn and all leaves are omitted.
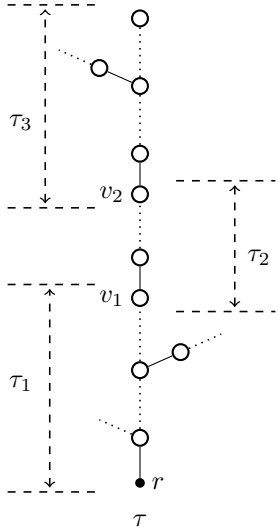
\begin{figure}[h]
\centering
\begin{tikzpicture}[scale=0.6]
  \tikzstyle{every node}=[draw,thick,circle,fill=white,minimum size=6pt, inner sep=1pt]
    \node (r) at (0,0) [fill=black,minimum size=0pt,label=right:$r$] {};
    \node (1) at (0,1)  {};
    \node (2) at (0,2.5)  {};
    \node (3a) at (0,4.1)  [label=left:$v_1$] {};
    \node (3b) at (0.9,2.9)   {};
    \node (4) at (0,5)  {};
    \node (5) at (0,6.4)  [label=left:$v_2$] {};
    \node (6) at (0,7.3) {};
    \node (7) at (0,8.8) {};
    \node (8a) at (0,10.3)  {};
    \node (8b) at (-0.9,9.2)  {};

    \draw  (r) -- (1);
    \draw [dotted,line width=0.6pt] (1)--(-0.9,1.4);
    \draw [dotted,line width=0.6pt] (1)--(2);
    \draw [dotted,line width=0.6pt](2)--(3a);
    \draw (3a)--(4);
    \draw [dotted,line width=0.6pt] (3b)--(1.8,3.3);
    \draw              (2)--(3b);
    \draw [dotted,line width=0.6pt](4)--(5);
    \draw (5)--(6);
    \draw [dotted,line width=0.6pt](6)--(7)--(8a);
    \draw (7)--(8b);
    \draw [dotted,line width=0.6pt](8b)--(-1.8,9.6);

    \draw [dashed,line width=0.6pt] (-1,4.4)--(-3,4.4);
    \draw [dashed,line width=0.6pt] (-1,-0.2)--(-3,-0.2);
    \draw [<->,dashed,line width=0.6pt] (-2,-0.1)--(-2,4.3);
    \node (tau1) at (-2.2,2.2) [minimum size=0pt,inner sep=0pt,label=left:$\tau_1$] {};

    \draw [dashed,line width=0.6pt] (0.8,3.8)--(3,3.8);
    \draw [dashed,line width=0.6pt] (0.8,6.7)--(3,6.7);
    \draw [<->,dashed,line width=0.6pt] (2.1,3.9)--(2.1,6.6);
    \node (tau1) at (2.2,5) [minimum size=0pt,inner sep=0pt,label=right:$\tau_2$] {};

    \draw [dashed,line width=0.6pt] (-1,6.1)--(-3,6.1);
    \draw [dashed,line width=0.6pt] (-1,10.6)--(-3,10.6);
    \draw [<->,dashed,line width=0.6pt] (-2.1,6.2)--(-2.1,10.5);
    \node (tau1) at (-2.2,8) [minimum size=0pt,inner sep=0pt,label=left:$\tau_3$] {};

\node[minimum size=0pt,inner sep=0pt,label=below:$\tau$] (name) at (0,-0.5){};
\end{tikzpicture}
\caption{Decomposition of $\tau$, where
$\h(v_1)=M_1+M_2$, $\h(v_2)=\h(\tau)-M_2$,
$v_1\in\Int(\tau_1)\cap\Int(\tau_2)$ and $v_2\in\Int(\tau_2)\cap\Int(\tau_3).$}
\label{xxfig7-xxlem4.10}
\end{figure}

Set $a:=\max\{\Ar(x):x\in \XX\}$ and $c:=|\XX|$. Then
\begin{align}
\label{E4.10.1}\tag{E4.10.1}
M_1+M_2=\h(\tau_1)\leq \wt(\tau_1)<M_3:=1+a+a^2+\cdots+a^{M_1+M_2}
\end{align}
and
\begin{align}
\label{E4.10.2}\tag{E4.10.2}
M_2+1=\h(\tau_3)\leq \wt(\tau_3)<M_3.
\end{align}
There are only finitely many, say $M_4$, PRTs of weight $<M_3$ and
thus there are at most $c^{M_3}\cdot M_4$ tree monomials of weight
$<M_3$. It follows from $\wt(\tau_1)+\wt(\tau_2)+\wt(\tau_3)=n+2$
and inequalities \eqref{E4.10.1} and \eqref{E4.10.2} that
\[
\h(\tau_2)=\wt(\tau_2)> n-2M_3+2.
\]
Thus, for all $n\geq d+2M_3-2$, we have $\h(\tau_2)> n-2M_3+2\geq d$.
By Lemma \ref{xxlem4.5}, there are at most $(d-1)^3$ possibilities
for $T_2$ for all $n\geq d+2M_3-2$. Since $T$ is uniquely determined
by its submonomials $T_1$, $T_2$ and $T_3$, there are at most
$(c^{M_3}\cdot M_4)^2(d-1)^3$ possibilities for $T$.
Therefore, there exist positive integers
\[
d_1:=(c^{M_3}\cdot M_4)^2(d-1)^3 \text{ and }
d_2:=d+2M_3-2
\]
such that $\Irr(I)$ contains at most $d_1$ tree monomials of weight
$h$ for all $h\geq d_2$.
\end{proof}

Now we are ready to prove our main result.

\begin{theorem}
\label{xxthm4.11}
Let $\XX$ be a finite operation alphabet and $\PP=\TX/\II$.
Write $\VV=\FF \XX$. Assume $\dim
(\sum_{i=0}^d \VV^i)/(\sum_{i=0}^{d-1} \VV^i)\leq d-3$ for
some $d\geq 3$ {\rm{(}}or $\GK(\PP)<2${\rm{)}}. Then there
exist positive real numbers $a,b$ such that
$$\dim \sum_{i=0}^n \PP(i)\leq an+b$$
for all $n\geq 0$. As a consequence $\GK(\PP)\leq 1$.
\end{theorem}

\begin{proof}
Define
\[
d_{\VV}(n):=\dim\sum_{i\leq n}\VV^i
\]
for all $n\in \NN$. By Lemma \ref{xxlem4.10}, there exist positive
integers $d_1$ and $d_2$ such that $\Irr(I)$ contains at most $d_1$
tree monomials of weight $h$ for all $h\geq d_2$.
Thus for all $n\geq d_2$, setting $n:=d_2+q$, we have that
\[
d_{\VV}(n)\leq d_{\VV}(d_2)+qd_1=d_{\VV}(d_2)+(n-d_2)d_1,
\]
and the function on the right hand side is linear in $n$.
Now the assertion follows from \eqref{E3.3.1}.
The consequence is clear.
\end{proof}

Theorem \ref{xxthm0.1} follows from Theorem \ref{xxthm4.11}
easily.


\section{Single-generated ns operads}
\label{xxsec5}

We say an operad $\PP$ is {\it single-generated} (or
{\it single-element generated}) if it is
generated by an operation alphabet $\XX$ that consists
of a single element. In this section, we will describe an
approach that can be used to construct a single-generated
single-branched operad from a finitely generated monomial
algebra. This procedure is called the {\it operadization} of
a finitely generated monomial algebra, which is different
from that given in Construction \ref{xxcon1.3}.

Fix an integer $d\geq 2$. Denote $X_0=\{x_1,x_2,\cdots,x_d\}$
and $X_0^*$ the free monoid generated by $X_0$. (Note that
$X_0$ is not the generating alphabet $\XX$ of an operad). Let
$\FF\langle X_0 \rangle: =\FF\langle x_1,x_2,\cdots,x_d\rangle
={\mathop{\bigoplus}\limits_{l\geq 0}} \FF\langle X_0 \rangle_{l}$
be the free graded algebra generated by $X_0$ with
$\deg(x_i)=1~(1\leq i\leq d)$.

We always assume that the tree monomials $T_d$ and
$R^\vee_{i,j}~(1\leq i<j\leq d)$ are defined as follows, see
Figure \ref{xxfig8-xxsec5}:
\begin{figure}[h]
\centering
\begin{tikzpicture}[scale=0.6]
  \tikzstyle{every node}=[draw,thick,circle,fill=white,minimum size=6pt, inner sep=1pt]
    \node (r) at (0,0) [fill=black,minimum size=0pt,label=right:$r$] {};
    \node (1) at (0,1.5)  {$a$};
    \node (2a) at (-1.5,3)  [fill=black,minimum size=0pt,label=above:$1$]{};
    \node (2b) at (0,3)  [fill=black,minimum size=0pt,label=above:$i$]{};
    \node (2c) at (1.5,3)  [fill=black,minimum size=0pt,label=above:$d$]{};
    \draw [dotted,line width=1pt] (-0.8,2.6)--(-0.4,2.6);
    \draw [dotted,line width=1pt] (0.4,2.6)--(0.8,2.6);
    \draw  (r) --(1)--(2a);
    \draw        (1)--(2b);
    \draw        (1)--(2c);
\node[minimum size=0pt,inner sep=0pt,label=below:$T_d$] (name) at (0,-0.5){};
\end{tikzpicture}
\hspace{8mm}
\begin{tikzpicture}[scale=0.6]
  \tikzstyle{every node}=[draw,thick,circle,fill=white,minimum size=6pt, inner sep=1pt]
    \node (r) at (0,0) [fill=black,minimum size=0pt,label=right:$r$] {};
    \node (1) at (0,1.5)  {$a$};
    \node (2a) at (-1,3) {$a$};
    \node (2b) at (1,3) {$a$};
    \node (i) at (-0.9,2.4) [minimum size=0pt,inner sep=0pt,label=right:$i$] {};
     \node (j) at (0.15,2.4) [minimum size=0pt,inner sep=0pt,label=right:$j$] {};

    \node (2x) at (-3,3)  [fill=black,minimum size=0pt,label=above:]{};
    \node (2z) at (-1.7,3)  [fill=black,minimum size=0pt,label=above:]{};
    \node (2u) at (-0.5,3)  [fill=black,minimum size=0pt,label=above:]{};
    \node (2v) at (0.5,3)  [fill=black,minimum size=0pt,label=above:]{};
    \node (2w) at (1.7,3)  [fill=black,minimum size=0pt,label=above:]{};
    \node (2y) at (3,3)  [fill=black,minimum size=0pt,label=above:]{};

    \node (3a) at (-1.8,4.5)  [fill=black,minimum size=0pt,label=above:]{};
    \node (3b) at (-1,4.5)  [fill=black,minimum size=0pt,label=above:]{};
    \node (3c) at (-0.2,4.5)  [fill=black,minimum size=0pt,label=above:]{};
    \node (3d) at (0.2,4.5)  [fill=black,minimum size=0pt,label=above:]{};
    \node (3e) at (1,4.5)  [fill=black,minimum size=0pt,label=above:]{};
    \node (3f) at (1.8,4.5)  [fill=black,minimum size=0pt,label=above:]{};

    \draw [dotted,line width=1pt] (-2,2.7)--(-1.6,2.7);
    \draw [dotted,line width=1pt] (-0.2,2.7)--(0.2,2.7);
    \draw [dotted,line width=1pt] (1.6,2.7)--(2,2.7);

    \draw [dotted,line width=1pt] (-1.5,4.2)--(-1.1,4.2);
    \draw [dotted,line width=1pt] (-0.9,4.2)--(-0.5,4.2);

    \draw [dotted,line width=1pt] (1.1,4.2)--(1.5,4.2);
    \draw [dotted,line width=1pt] (0.5,4.2)--(0.9,4.2);

    \draw  (r) -- (1)--(2x);
    \draw         (1)--(2a);
    \draw         (1)--(2b);
    \draw         (1)--(2y);
    \draw         (1)--(2z);
    \draw         (1)--(2w);
    \draw         (1)--(2u);
    \draw         (1)--(2v);
    \draw              (2a)--(3a);
    \draw              (2a)--(3b);
    \draw              (2a)--(3c);
     \draw             (2b)--(3d);
    \draw              (2b)--(3e);
    \draw              (2b)--(3f);
\node[minimum size=0pt,inner sep=0pt,label=below:$R^\vee_{i,j}$] (name) at (0,-0.5){};
\end{tikzpicture}
\caption{}
\label{xxfig8-xxsec5}
\end{figure}\\
Suppose that $\widetilde{\QQ}$ is the operad
presented by generator $\XX=\{a\}$ and relations
$\{(a\circ_{j}a)\circ_{i}a,~1\leq i<j\leq d\}$, then
$\widetilde{\QQ}$ is a single-branched operad. We
also say $\widetilde{\QQ}$ is the operad
generated by $T_d$ and subject to the relations
$R^\vee_{i,j}~(1\leq i<j\leq d)$.

Identifying the tree monomial $T_d$ and its
labelling $a$, we define the following
$\FF$-linear \emph{operadization map}, which is
an isomorphism of vector spaces,
\begin{eqnarray*}\label{}
\mathcal{O}: \FF\langle X_0 \rangle \longrightarrow \widetilde{\QQ}_{\geq 2}
\end{eqnarray*}
by setting
\begin{eqnarray*}\label{}
\mathcal{O}(w)=\mathcal{O}(x_{i_1} x_{i_2} \cdots x_{i_k})
=a \circ_{i_1} a \circ_{i_2} \cdots \circ_{i_k} a
:=(a \circ_{i_1} (a \circ_{i_2}( \cdots \circ_{i_k} a)))
\end{eqnarray*}
for any $w=x_{i_1} x_{i_2} \cdots x_{i_k} \in X_0^*$.
See \eqref{E4.0.1} and \eqref{E4.0.2}.

Now we are ready to give the following construction,
in the proof of which we describe an operadization
of a finitely generated monomial algebra.

\begin{construction}
\label{xxcon5.1}
For any finitely generated monomial algebra
$A=\FF\langle X_0 \rangle/I$ with
$I\subseteq {\mathop{\bigoplus}\limits_{l\geq 2}}
\FF\langle X_0 \rangle_{l}$ being a monomial ideal of
the free algebra, there exists a single-generated
single-branched ns operad $\QQ_A$ such that
\[\GK (\QQ_A)=\GK (A).\]
\end{construction}

\begin{proof}
Since $I$ is a monomial ideal of $\FF\langle X \rangle$,
it is routine to check that $\mathcal{O}(I)$ is an ideal
of $\widetilde{\QQ}$.

Define the operad $\QQ_A$ to be the quotient operad
$\widetilde{\QQ}/\mathcal{O}(I)$. Noting that
\begin{eqnarray*}\label{}
\dim \QQ_A(n)=
\left\{
 \begin{array}{ll}
1,~\quad~&{\rm when}~n=1~{\rm or}~d,\\
\dim A_l,~\quad~&{\rm when}~n=(l+1)d-l~{\rm and}~l\geq 1,\\
0,~\quad~&{\rm otherwise},
\end{array}
\right.
\end{eqnarray*}
we obtain that
\[\GK (\QQ_A)=\GK (A).\]
\end{proof}

\begin{definition}
\label{xxdef5.2}
Suppose $d:=|X_0|\geq 2$. The single-generated single-branched
ns operad $\QQ_A=\widetilde{\QQ}/\mathcal{O}(I)$ defined in
Construction $\ref{xxcon5.1}$ is called the {\it operadization}
of the finitely generated monomial algebra
$A=\FF\langle X_0 \rangle/I$, where
$I\subseteq {\mathop{\bigoplus}\limits_{l\geq 2}} \FF\langle X_0 \rangle_{l}$
is a monomial ideal of the free algebra.
\end{definition}

Note that Construction \ref{xxcon5.1} is related to some ideas
presented in \cite{DMR20, DT20}. For example, a special case of
this construction is given in \cite[Section 7.1.4]{DT20}.

Next we will present some examples of single-generated
single-branched operads by using the operadization
procedure described above.

\begin{example}
\label{xxex5.3}
Let $d=2$. Suppose
$T_2,R^\vee_{1,2},R_{1,1},R_{1,2},R_{2,1},R_{2,2}$
are defined as in Figure \ref{xxfig9-xxex5.3}.
\begin{figure}[h]
\centering
\begin{tikzpicture}[scale=0.6]
  \tikzstyle{every node}=[draw,thick,circle,fill=white,minimum size=6pt, inner sep=1pt]
    \node (r) at (0,0) [fill=black,minimum size=0pt,label=right:$r$] {};
    \node (1) at (0,1)  {$a$};
    \node (2a) at (-1,2)  [fill=black,minimum size=0pt,label=above:$1$]{};
    \node (2b) at (1,2)  [fill=black,minimum size=0pt,label=above:$2$]{};
    \draw  (r) --(1)--(2a);
    \draw        (1)--(2b);
\node[minimum size=0pt,inner sep=0pt,label=below:$T_2$] (name) at (0,-0.5){};
\end{tikzpicture}
\hspace{8mm}
\begin{tikzpicture}[scale=0.6]
  \tikzstyle{every node}=[draw,thick,circle,fill=white,minimum size=6pt, inner sep=1pt]
    \node (r) at (0,0) [fill=black,minimum size=0pt,label=right:$r$] {};
    \node (1) at (0,1)  {$a$};
    \node (2a) at (-1,2) {$a$};
    \node (2b) at (1,2) {$a$};
    \node (3a) at (-1.5,3)  [fill=black,minimum size=0pt,label=above:$1$]{};
    \node (3b) at (-0.5,3)  [fill=black,minimum size=0pt,label=above:$2$]{};
    \node (3c) at (0.5,3)  [fill=black,minimum size=0pt,label=above:$3$]{};
    \node (3d) at (1.5,3)  [fill=black,minimum size=0pt,label=above:$4$]{};
    \draw  (r) -- (1)--(2a)--(3a);
    \draw         (1)--(2b)--(3c);
    \draw              (2a)--(3b);
    \draw              (2b)--(3d);
\node[minimum size=0pt,inner sep=0pt,label=below:$R^\vee_{1,2}$] (name) at (0,-0.5){};
\end{tikzpicture}
\hspace{8mm}
\begin{tikzpicture}[scale=0.6]
  \tikzstyle{every node}=[draw,thick,circle,fill=white,minimum size=6pt, inner sep=1pt]
    \node (r) at (0,0) [fill=black,minimum size=0pt,label=right:$r$] {};
    \node (1) at (0,1)  {$a$};
    \node (2a) at (-1,2) {$a$};
    \node (2b) at (1,2) [fill=black,minimum size=0pt,label=above:$4$]{};
    \node (3a) at (-2,3) {$a$};
    \node (3b) at (-0,3)  [fill=black,minimum size=0pt,label=above:$3$]{};
    \node (4a) at (-3,4)  [fill=black,minimum size=0pt,label=above:$1$]{};
    \node (4b) at (-1,4)  [fill=black,minimum size=0pt,label=above:$2$]{};
    \draw  (r) -- (1)--(2a)--(3a)--(4a);
    \draw         (1)--(2b);
    \draw              (2a)--(3b);
    \draw              (3a)--(4b);
\node[minimum size=0pt,inner sep=0pt,label=below:$R_{1,1}$] (name) at (0,-0.5){};
\end{tikzpicture}
\\
\vspace{0.6cm}
\begin{tikzpicture}[scale=0.6]
  \tikzstyle{every node}=[draw,thick,circle,fill=white,minimum size=6pt, inner sep=1pt]
    \node (r) at (0,0) [fill=black,minimum size=0pt,label=right:$r$] {};
    \node (1) at (0,1)  {$a$};
    \node (2a) at (-1,2) {$a$};
    \node (2b) at (1,2) [fill=black,minimum size=0pt,label=above:$4$]{};
    \node (3a) at (-2,3)  [fill=black,minimum size=0pt,label=above:$1$]{};
    \node (3b) at (-0,3)  {$a$};
    \node (4b) at (-1,4)  [fill=black,minimum size=0pt,label=above:$2$]{};
    \node (4c) at (1,4)  [fill=black,minimum size=0pt,label=above:$3$]{};
    \draw  (r) -- (1)--(2a)--(3a);
    \draw         (1)--(2b);
    \draw              (2a)--(3b)--(4c);
    \draw         (3b)--(4b);
\node[minimum size=0pt,inner sep=0pt,label=below:$R_{1,2}$] (name) at (0,-0.5){};
\end{tikzpicture}
\hspace{8mm}
\begin{tikzpicture}[scale=0.6]
  \tikzstyle{every node}=[draw,thick,circle,fill=white,minimum size=6pt, inner sep=1pt]
    \node (r) at (0,0) [fill=black,minimum size=0pt,label=right:$r$] {};
    \node (1) at (0,1)  {$a$};
    \node (2a) at (-1,2) [fill=black,minimum size=0pt,label=above:$1$] {};
    \node (2b) at (1,2) {$a$};
    \node (3a) at (0,3) {$a$};
    \node (3b) at (2,3)  [fill=black,minimum size=0pt,label=above:$4$] {};
    \node (4b) at (-1,4)  [fill=black,minimum size=0pt,label=above:$2$] {};
    \node (4c) at (1,4)  [fill=black,minimum size=0pt,label=above:$3$] {};
    \draw  (r) -- (1)--(2a);
    \draw         (1)--(2b)--(3a)--(4c);
    \draw              (2b)--(3b);
    \draw              (3a)--(4b);
\node[minimum size=0pt,inner sep=0pt,label=below:$R_{2,1}$] (name) at (0,-0.5){};
\end{tikzpicture}
\hspace{8mm}
\begin{tikzpicture}[scale=0.6]
  \tikzstyle{every node}=[draw,thick,circle,fill=white,minimum size=6pt, inner sep=1pt]
    \node (r) at (0,0) [fill=black,minimum size=0pt,label=right:$r$] {};
    \node (1) at (0,1)  {$a$};
    \node (2a) at (-1,2) [fill=black,minimum size=0pt,label=above:$1$] {};
    \node (2b) at (1,2) {$a$};
    \node (3a) at (0,3)  [fill=black,minimum size=0pt,label=above:$2$] {};
    \node (3b) at (2,3)  {$a$};
    \node (4a) at (1,4)  [fill=black,minimum size=0pt,label=above:$3$] {};
    \node (4b) at (3,4)  [fill=black,minimum size=0pt,label=above:$4$] {};
    \draw  (r) -- (1)--(2a);
    \draw         (1)--(2b)--(3a);
    \draw              (2b)--(3b)--(4b);
    \draw              (3b)--(4a);
\node[minimum size=0pt,inner sep=0pt,label=below:$R_{2,2}$] (name) at (0,-0.5){};
\end{tikzpicture}
\caption{}
\label{xxfig9-xxex5.3}
\end{figure}\\

\begin{enumerate}
\item[(1)]
If $\QQ$ is the operad generated by $T_2$ and subject
to the relation $R^\vee_{1,2}$, then we can check that
${\rm dim}\QQ(n)=2^{n-2}$ for $n\geq 2$. Hence in this
case $G_{\QQ}(z)=z+\frac{z^2}{1-2z}$ and $\GK(\QQ)=\infty$.
\item[(2)]
Define $\QQ$ to be the operad generated by $T_2$ and
subject to the relations $R^\vee_{1,2}, R_{1,1}$.
Then we can check that $\QQ$ is an operadization of the
graded algebra $A:=\FF\langle x_1,x_2 \rangle/\langle x^2_1 \rangle$.
and
\begin{eqnarray*}
{\rm dim}\QQ(n)={\rm dim}\QQ(n-1) + {\rm dim}\QQ(n-2)~ {\rm for}~ n\geq 2.
\end{eqnarray*}
Hence in this case $G_{\QQ}(z)=\frac{z}{1-z+z^2}$,
$\GK(\QQ)=\infty$, and we call $\QQ$ the \emph{Fibonacci operad}.
\item[(3)]
Define $\QQ$ to be the operad generated by $T_2$ and subject
to the relations $R^\vee_{1,2}, R_{2,1},R_{2,2}$. Then
$\QQ$ is an operadization of the graded algebra
$\FF\langle x_1,x_2 \rangle/\langle x_2 x_1, x^2_2 \rangle$.
and
\begin{eqnarray*}
{\rm dim}\QQ(n)=
\left\{
 \begin{array}{ll}
1,~\quad~&{\rm when}~n=1~{\rm or}~2,\\
2,~\quad~&{\rm when}~n\geq 3,\\
0,~\quad~&{\rm when}~n=0.
\end{array}
\right.
\end{eqnarray*}
Hence $G_{\QQ}(z)=z+z^2+\frac{2z^3}{1-z}$ and $\GK(\QQ)=1$.
\end{enumerate}
\end{example}

If we consider a graded algebra $A=\FF\langle x_1,x_2 \rangle
/\langle x_1x_2-x^2_1 \rangle$ which is not a monomial algebra.
It is easy to check that ${\mathcal O}(I)$ is not an ideal of
$\widetilde{\QQ}$. So it requires more work to understand the
quotient operad $\widetilde{\QQ}/\langle {\mathcal O}(I) \rangle$
in this case.

Now we are ready to prove Theorem \ref{xxthm0.2}.

\begin{proof}[Proof of Theorem \ref{xxthm0.2}]
(1) Let $\PP$ be a finitely generated locally finite
ns operad. It is clear that $\GK(\PP)\geq 0$.
By Proposition \ref{xxpro3.8}, $\GK(\PP)$ is not
strictly between $0$ and $1$. By Theorem \ref{xxthm0.1},
$\GK(\PP)$ is not between $1$ and $2$. Therefore
$\GK(\PP)\in R_{\GK}$, see \eqref{E0.0.1}
for the definition of $R_{\GK}$.

(2) Let $r\in R_{\GK}$. Let $A$ be a finitely generated algebra
of GK-dimension $r$. By \cite[Remark 4.1]{Bel15}, we
can assume $A$ is a monomial algebra finitely generated in
degree 1. By Construction \ref{xxcon5.1}, $\QQ_A$ is a single-generated
single-branched locally finite operad with $\GK(\QQ_A)=r$.
The assertion follows.
\end{proof}

\section{Generating series and exponential generating series}
\label{xxsec6}

The generating series of an operad is defined in \eqref{E0.2.1}.
By \cite[(0.1.1)]{KP15}, the {\it exponential generating series}
of an operad $\PP$ is defined to be
\begin{equation}
\label{E6.0.1}\tag{E6.0.1}
E_{\PP}(z):=\sum_{n\geq 0}\frac{\dim \PP(n)}{n!} z^n.
\end{equation}

Clearly, $G_{\PP}(z)$ (resp. $E_{\PP}(z)$) contains more
information than $\GK(\PP)$. A list of (exponential) generating
series of operads are given in \cite{Zin12, KP15}. Several authors have
recently been studying the holonomic and differential algebraic
properties of $G_{\PP}(z)$ (resp. $E_{\PP}(z)$), see Definition
\ref{xxdef0.3}.

By \cite[Theorem 1.5]{Sta80},
$F(z):=\sum_{n=0}^{\infty} f(n) z^n$ is holonomic if
and only if the sequence $\{f(n)\}_{n\geq 0}$ satisfies a
recurrence relation of the form
\begin{equation}
\label{E6.0.2}\tag{E6.0.2}
f(n)=q_1(n) f(n-1) +\cdots + q_{n-m}(n) f(n-m), \quad n\gg 0
\end{equation}
for some fixed $m$ and rational functions $q_1(x),\cdots,q_m(x)$.
For all examples given in Example \ref{xxex5.3}, $G_{\PP}(z)$
are rational, and consequently, holonomic. Using \eqref{E6.0.2},
it is easy to see that
\begin{equation}
\label{E6.0.3}\tag{E6.0.3}
{\text{$G_{\PP}(z)$ is holonomic if and only if so is $E_{\PP}(z)$.}}
\end{equation}

For two integers $a$ and $b$, let $[a,b]_{\NN}:=[a,b]\cap \NN$.
Suppose the sequence $\{f(n)\}_{n\geq 0}$ has infinitely many
nonzeros. Let
$$\Phi_{F}=\{n \mid f(n)= 0\}$$
and write $\Phi$ as $[i_1, j_1]_{\NN}\cup [i_2, j_2]_{\NN}\cup \cdots$
where $i_s\leq j_s \leq i_{s+1}-2$ for all $s\geq 1$. Then
the following is true.

\begin{lemma}
\label{xxlem6.1}
Let $F(z)=\sum_{n=0}^{\infty} f(n) z^n$ with infinitely many
nonzero coefficients. Write $\Phi_{F}=\bigcup_{s\geq 1}
[i_s,  j_s]_{\NN}$. Suppose that $\limsup_{s\to\infty} (j_s-i_s)=\infty$.
Then $F(z)$ is not holonomic.
\end{lemma}

\begin{proof} Suppose to the contrary that $F(z)$ is holonomic.
Then \eqref{E6.0.2} holds for some $m$. Pick any $s$ such that
$j_s-i_s\geq m+1$. By the definition of $\Phi_{F}$, $f(n)=0$ for
all $i_s\leq n\leq j_s$. By \eqref{E6.0.2} and induction, we have
that $f(n)=0$ for all $n\geq i_s+m$. Therefore $F(z)$ has
only finitely many nonzero coefficients, a contradiction.
\end{proof}

The next example shows that we can easily construct a monomial
algebra such that its Hilbert series is not holonomic.

\begin{example}
\label{xxex6.2}
Let $U$ be the monomial algebra $\FF\langle x_1,x_2\rangle/I$
where $\deg x_1=\deg x_2=1$ and $I$ is a span of monomials of
the form
\begin{enumerate}
\item[(a)]
all monomials having degree $\geq 3$ in $x_2$,
\item[(b)]
$x_1^i x_2 x_1^j $ with $i>0$ or $j>0$,
\item[(c)]
$x_2 x_1^i x_2$ where $(2m+1)^{2m+1}-1 \leq i\leq (2m+2)^{2m+2}-1$
for some integer $m\geq 0$.
\end{enumerate}
Let $\Lambda$ be $\bigcup_{m=0}^{\infty}
[(2m+1)^{2m+1}+1, (2m+2)^{2m+2}+1]_{\NN}$
which is a subset of $\NN$. Let $\Lambda^c=\NN\setminus \Lambda$.
Then $\Lambda^c=\{0\}\cup \{1\} \cup \bigcup_{m=0}^{\infty}
[(2m+2)^{2m+2}+2, (2m+3)^{2m+3}]_{\NN}$.
Define
$$\delta_{\Lambda}(n)=\begin{cases}
0 & n\in \Lambda,\\
1 & n\in \Lambda^c.
\end{cases}$$
Then $U$ has an $\FF$-basis $\{x_1^i\}_{i\geq 0}
\cup \{x_1^i x_2\}_{i\geq 0}\cup\{x_2 x_1^i\}_{i\geq 0}
\cup \{x_2 x_1^i x_2\}_{i+2\in \Lambda^c}$ and
its Hilbert series is
$$H_U(z)=1+ 2z+ \sum_{n=2}^{\infty} (3+\delta_{\Lambda}(n)) z^n
=1+2 z +\frac{3z^2}{(1-z)}+V(z)$$
where $V(z)=\sum_{n=2}^{\infty} \delta_{\Lambda}(n) z^n$.
Since $V(z)$ satisfies the hypothesis of Lemma \ref{xxlem6.1},
it is not holonomic.

Let $\PP_U$ be the construction given in Construction \ref{xxcon1.3}.
It is easy to see that
$$\GK(\PP_U)=\GK(U)=1$$
and that
$$G_{\PP_U}(z)=z H_U(z)=z(1+2 z +\frac{3z^2}{(1-z)})+zV(z)$$
where $z(1+2 z +\frac{3z^2}{(1-z)})$ is holonomic,
but $zV(z)$ is not (by Lemma \ref{xxlem6.1} again).
Therefore $G_{\PP_U}(z)$ is not holonomic
by \cite[Holonomic Theorem 2]{Ber14}.

Let $\QQ_U$ be the construction given in Construction \ref{xxcon5.1}.
Then $\GK(\QQ_U)=1$ and
$$\begin{aligned}
G_{\QQ_U}(z)&=z+ z^d H_U(z^{d-1})=z+z^2H_U(z)\\
&= z+ z^2 (1+2 z +\frac{3z^2}{(1-z)}) +z^2 V(z)
\end{aligned}$$
as $d=\dim U_1=2$. By an argument similar to the previous
paragraph, $G_{\QQ_U}(z)$ is not holonomic.
\end{example}

Now we are ready to prove Proposition \ref{xxpro0.5}.

\begin{proof}[Proof of Proposition \ref{xxpro0.5}]
Let $r\in R_{\GK}\setminus \{0\}$. We claim that
there is a monomial algebra $B$ finitely generated
in degree 1 such that $\GK(B)=r$ and that $H_B(z)$
is not holonomic. By \cite[Remark 4.1]{Bel15}, there
is a monomial algebra $A$ finitely generated
in degree 1 such that $\GK(A)=r$. If $H_A(z)$ is
not holonomic, then we are done by setting $B=A$.
If $H_A(z)$ is holonomic, let $B=\FF \oplus (A_{\geq 1} \oplus
U_{\geq 1})$. Then $B$ is a monomial algebra
finitely generated in degree 1 such that $\GK(B)=r$
and that $H_B(z)=H_A(z)+H_U(z)-1$. Since $H_U(z)$ is not
holonomic, by \cite[Holonomic Theorem 2]{Ber14},
$H_{B}(z)$ is not holonomic. So we proved the claim.

Let $\PP=\QQ_B$ as in Construction \ref{xxcon5.1}. Then
$\PP$ is single-generated, $\GK(\PP)=\GK(B)=r$ and
$G_{\PP}=z+z^d H_B(z^{d-1})$.
Since $H_B(z)$ is not not holonomic, it is routine to show
that $G_{\PP}(z)$ is not holonomic.
\end{proof}

The following lemma is easy and the proof is analogous to
the one given in the quadratic case, see the proof
of \cite[Corollary 4.2(i)]{Dot19}.

\begin{lemma}
\label{xxlem6.3}
Let $A$ be a connected graded locally finite
algebra and let $\PP_A$ be the operad constructed in
Construction \ref{xxcon1.3}.
\begin{enumerate}
\item[(1)]
$A$ is finitely generated if and only if so is $\PP_A$.
\item[(2)]
$A$ is finitely presented if and only if so is $\PP_A$.
\end{enumerate}
\end{lemma}

For the rest of this section, we construct a finitely
presented locally finite ns operad such that its
generating series is not holonomic. Therefore we give
a ``non-generic'' counterexample to
\cite[Expectation 2]{KP15} (see Expectation \ref{xxexp0.4}).

Let us recall some history in the setting of
noncommutative graded algebra. In 1972 Govorov
conjectured that the Hilbert series of a finitely
presented graded algebra is rational \cite{Gov72}.
This was shown to be false, for example, by Shearer
in \cite{She80} by constructing an irrational (but
algebraic) Hilbert series of a finitely generated graded
algebra. Shearer mentioned a similar construction
giving also an example with a transcendental (but still
holonomic) Hilbert series. His third example,
involving the generating function of the number of
partitions, indeed has its Hilbert series not holonomic.
A similar example was given in \cite{Smi76}
which contains the following example as a special case.

\begin{example}
\label{xxex6.4}
Suppose ${\rm{char}}\; \FF=0$.
Let $L$ be the graded Lie algebra with basis $\{e_1,
e_2,\cdots, e_n,\cdots\}$ with $\deg e_i=i$ for all $i$
and Lie bracket determined by
$$[e_i, e_j]=(i-j) e_{i+j}, \quad \forall \; i\neq j.$$
This is a subalgebra of the Witt algebra.
Let $A$ be the universal enveloping algebra of $L$.
Then $A$ has intermediate growth and is generated by
$e_1$ and $e_2$ ($\deg e_i=i$) and subject to the
relations $e_3 e_2-e_2e_3=e_5$ and $e_5 e_2-e_2e_5=3e_7$.
So $A$ is finitely presented, but not generated in
degree 1.

It is easy to see that
$$H_A(z)=\prod_{i=1}^{\infty} \frac{1}{(1-z^i)}.$$
Note that $H_A(z)$ is equal to $P(z):=\sum_{n=0}^{\infty} p(n) z^n$
where $p(n)$ is the number of partitions of $n$.

Now we list some facts about $P(z)$:
\begin{enumerate}
\item[(P1)]\cite[p.187]{Sta80}
$P(z)$ is not holonomic.
\item[(P2)]
$zP(z)$ is not holonomic.
\item[(P3)]
$Q(z):=z(zP(z))'(=\sum_{n=0}^{\infty} p(n) (n+1)z^{n+1})$
is not holonomic.
\end{enumerate}

Note that Parts (P2,P3) follow from part (P1) immediately.

Now let $\PP_A$ be the operad given by Construction \ref{xxcon1.3}.
By Lemma \ref{xxlem6.3}, $\PP_A$ is finitely generated.
By \eqref{E1.3.2}
$$G_{\PP_A}(z)=zH_{A}(z)=zP(z)$$
which is not holonomic by (P2). Therefore
$\PP_A$ is a ``non-generic'' counterexample to
\cite[Expectation 2]{KP15}.

Note that there is a quadratic algebra with similar properties
on its Hilbert series, but with 14 generators and 96 quadratic
relations \cite[Theorem 1(iv)]{Koc15}. The algebra $A$ could be
replaced by the Kocak's example.
\end{example}

\begin{remark}
\label{xxrem6.5}
Using the same proof as above, one sees that
a version of Proposition \ref{xxpro0.5}
holds for exponential generating series. See \eqref{E6.0.1}
for the definition of the exponential generating series.
\end{remark}

\section{Comments on symmetric operads}
\label{xxsec7}

In this final section we make some comments on symmetric operads
and prove the main result of this section, namely, Theorem
\ref{xxthm0.7}.

First of all we refer to \cite[Chapter 5]{LV12} or
\cite[Definition 1.2]{BYZ20} for the definition of a symmetric
operad.

Given a graded algebra with augmentation, there is an easy
construction of a symmetric operad similar to the one given
in Construction \ref{xxcon1.3}.

\begin{construction}
\label{xxcon7.1}
Let $A:=\oplus_{i\geq0} A_i$ be a locally finite $\NN$-graded
algebra with unit $1_A$. Suppose that $A$ has a graded
augmentation $\epsilon: A\to \FF$ such that $\fm:=\ker \epsilon$
is a maximal graded ideal of $A$. We let $C_n$ be the cyclic group
of order $n$ and elements in $C_n$ are denoted by $\{1,\cdots,n\}$.

Let $\SO_A(0)=0$ and $\SO_A(n)=A_{n-1}\otimes \FF C_n$ for all
$n\geq 1$. Elements in $\SO_A(n)$ are $\FF$-linear combinations
of $(a,i)$ where $a\in A_{n-1}$ and $1\leq i\leq n$. The $\SG$-action
on $\SO_A(n)$ is determined by
$$(a,i)\ast \sigma=(a, \sigma^{-1}(i))$$
for all $\sigma \in \SG_n$.

Define partial compositions as follows, for $1\leq s\leq m$,
$$\begin{aligned}
\circ_s: \quad & \quad \SO_A(m)\otimes \SO_A(n)\to \SO_A(n+m-1),\\
&\quad (a_{m-1}, i)  \otimes (a_{n-1}, j) \mapsto
\begin{cases}
(c a_{m-1}, i) &a_{n-1}=c1_A, c\in \FF,\\
(a_{m-1}a_{n-1}, i+j-1) & a_{n-1}\in \fm, s=i,\\
0 & a_{n-1}\in \fm, s\neq i.\end{cases}
\end{aligned}
$$
We claim that that $\SO_A:=\{\SO_A(n)\}_{n\geq 0}$
is a symmetric operad with identity $\id=(1_A,1)$. We give
a sketch proof below. 

We refer to \cite[Section 5.3.4]{LV12} for the partial
definition of a symmetric operad. In fact, a symmetric
operad $\PP$ is an $\SG$-module satisfying
the axioms of a nonsymmetric operad [Definition \ref{xxdef1.1}]
and the following two additional equations
\begin{align}
\label{E7.1.1}\tag{E7.1.1}
\mu \circ_s (\nu\ast \sigma) &= (\mu\circ_s \nu)\ast \sigma'\\
\label{E7.1.2}\tag{E7.1.2}
(\mu\ast \phi)\circ_s \nu&=(\mu\circ_{\phi(s)} \nu)\ast \phi''
\end{align}
where $\mu\in \PP(m), \nu\in \PP(n), 1\leq s\leq m, \sigma\in \SG_{n},
\phi\in \SG_{m}$ and where $\sigma'=\bfone_m \circ_s \sigma$ and
$\phi''=\phi\circ_s \bfone_n$. See \cite[Section 5.3.4]{LV12} and
\cite[(E1.2.1) and (E8.1.3)]{BYZ20} for the explanation of
$\sigma'$ and $\phi''$.
We first verify \eqref{E7.1.1} and \eqref{E7.1.2} and
then the rest of axioms given in Definition \ref{xxdef1.1}
for $\PP:={\mathcal S}_A$.

\medskip
\noindent
{\bf Verification of \eqref{E7.1.1}:} Write $\mu=(a_{m-1}, i)$ and
$\nu =(a_{n-1},j)$. If $n=1$ and $a_{n-1}=c 1_A$ (or if $n=1$
and $a_{n-1}\in \fm$), then $j=1$ and $\sigma={\bfone}_1\in \SG_1$ and
$\sigma'={\bfone}_{n+m-1}\in \SG_{n+m-1}$. Clearly
\eqref{E7.1.1} holds. If $n\geq 2$ and $s\neq i$, both sides of
\eqref{E7.1.1} are zero. It remains to consider the case
when $n\geq 2$ and $s=i$. Write
\begin{equation}
\label{E7.1.3}\tag{E7.1.3}
\sigma=\begin{pmatrix} k_1 & k_2 &\cdots & k_n\\
1&2&\cdots& n\end{pmatrix}
\end{equation}
where by convention $k_i=\sigma^{-1}(i)$ for all $i$.
Then, by definition,
$${\small \sigma'=\begin{pmatrix} 1 &\cdots &i-1 & k_1+i-1 & k_2+i-1 &\cdots & k_n+i-1 &
i+n&\cdots &n+m-1\\
1 &\cdots & i-1 & i& i+1 &\cdots & n+i-1 &
i+n&\cdots & n+m-1\end{pmatrix}.}$$
In this case, we have
$$\begin{aligned}
{\text{LHS of \eqref{E7.1.1}}}
&=(a_{m-1},i)\circ_i [(a_{n-1}, j)\ast \sigma]
 =(a_{m-1},i)\circ_i (a_{n-1}, \sigma^{-1}(j))\\
&=(a_{m-1}a_{n-1}, i+\sigma^{-1}(j)-1)
 =(a_{m-1}a_{n-1}, i+k_j-1)\\
{\text{RHS of \eqref{E7.1.1}}}
&=[(a_{m-1},i)\circ_i (a_{n-1}, j)]\ast \sigma'
 =(a_{m-1}a_{n-1}, i+j-1)\ast \sigma'\\
&=(a_{m-1}a_{n-1}, (\sigma')^{-1}(i+j-1))
 =(a_{m-1}a_{n-1}, k_j+i-1).
\end{aligned}$$
Therefore \eqref{E7.1.1} holds.

\medskip
\noindent
{\bf Verification of \eqref{E7.1.2}:} Recycle the letter $k_i$ and write
$$\phi=\begin{pmatrix} k_1 & k_2 &\cdots & k_m\\
1&2&\cdots& m\end{pmatrix}
$$
using the convention of \eqref{E7.1.3}. If $a_{n-1}=c1_A$, then
$n=j=1$ and $\phi''=\phi\circ_s \bfone_1=\phi$. Consequently,
$\phi^{-1}(i)=k_i=(\phi'')^{-1}(i)$. In this case,
$$\begin{aligned}
{\text{LHS of \eqref{E7.1.2}}}
&=[(a_{m-1},i)\ast \phi]\circ_s (c1_A, 1)
 =(a_{m-1},\phi^{-1}(i))\circ_s (c1_A, 1)\\
&=(c a_{m-1},\phi^{-1}(i))=(c a_{m-1},k_i)\\
{\text{RHS of \eqref{E7.1.2}}}
&=[(a_{m-1},i)\circ_{\phi(s)} (c1_A,1)]\ast \phi''
 =(ca_{m-1}, i)\ast \phi''\\
&=(ca_{m-1}, i)\ast \phi=(ca_{m-1}, k_i).
\end{aligned}$$
Hence \eqref{E7.1.2} holds. Next we assume that $a_{n-1}\in \fm$.
If $s\neq \phi^{-1}(i)$, then both sides of
\eqref{E7.1.2} are zero. It remains to consider the case
when $s=\phi^{-1}(i)$ or $i=\phi(s)$. By definition,
$${\small
\phi''=
\begin{pmatrix}
k'_1 & \cdots & k'_{i-1} & k'_{i} & k'_{i}+1 &\cdots
& k'_i+n-1& k'_{i+1}& \cdots & k'_m\\
1    & \cdots & i-1      & i      &i+1       &\cdots
&i+n-1    & i+n     & \cdots & m+n-1
\end{pmatrix}
}$$
where $k'_{t}=\begin{cases} k_t+n-1 & k_t>s \\
k_t & k_t \leq s\end{cases}$. In particular,
$$(\phi'')^{-1}(i+j-1)=k'_{i}+j-1=s+j-1=k_i+j-1=\phi^{-1}(i)+j-1.$$
Now we compute
$$\begin{aligned}
{\text{LHS of \eqref{E7.1.2}}}
&=[(a_{m-1},i)\ast \phi]\circ_s (a_{n-1}, j)
 =[(a_{m-1},i)\ast \phi]\circ_{\phi^{-1}(i)} (a_{n-1}, j)\\
&=(a_{m-1},\phi^{-1}(i))\circ_{\phi^{-1}(i)} (a_{n-1}, j)
 =(a_{m-1}a_{n-1}, \phi^{-1}(i)+j-1)\\
{\text{RHS of \eqref{E7.1.2}}}
&=[(a_{m-1},i)\circ_{\phi(s)} (a_{n-1}, j)]\ast \phi''
 =[(a_{m-1},i)\circ_{i} (a_{n-1}, j)]\ast \phi''\\
&=(a_{m-1}a_{n-1}, i+j-1)\ast \phi''
 =(a_{m-1}a_{n-1}, (\phi'')^{-1}(i+j-1))\\
&=(a_{m-1}a_{n-1}, \phi^{-1}(i)+j-1)
\end{aligned}
$$
which implies that \eqref{E7.1.2} holds in this case.
So we verified \eqref{E7.1.2} for all cases.

\medskip
\noindent
{\bf Verification of \eqref{E1.1.3}:} It follows easily from the
definition that
$$(1_A, 1)\circ_1 (a_{n-1},j)=(a_{n-1}, j)
=(a_{n-1}, j)\circ_s (1_A, 1)$$
for all $1\leq s\leq n$, which is \eqref{E1.1.3}.

\medskip
\noindent
{\bf Verification of \eqref{E1.1.2}:} Recall that
\eqref{E1.1.2} is equivalent to
\begin{equation}
\label{E7.1.4}\tag{E7.1.4}
(\lambda \circ_{s} \mu) \circ_{t-1+m} \nu =(\lambda\circ_t \nu)\circ_{s}\mu
\end{equation}
for $\lambda\in \PP(l)$, $\mu\in \PP(m)$, $\nu\in
\PP(n)$ and $1\leq s<t \leq l$.
Write $\lambda=(a_{l-1}, i)$, $\mu=(a_{m-1}, j)$ and $\nu=(a_{n-1},k)$.
If the LHS of \eqref{E7.1.4} were nonzero, we must have $s=i$ and
$s+j-1=t-1+m$. But $s<t$ and $j\leq m$ which contradict
the equation $s+j-1=t-1+m$. So the LHS of \eqref{E7.1.4}
is zero. For a similar reason, the RHS of \eqref{E7.1.4} is zero.
Hence \eqref{E7.1.4} holds.

\medskip
\noindent
{\bf Verification of \eqref{E1.1.1}:} Recall that
\eqref{E1.1.1} is equivalent to
\begin{equation}
\label{E7.1.5}\tag{E7.1.5}
(\lambda \circ_{s} \mu) \circ_{s-1+t} \nu =\lambda\circ_s (\mu\circ_{t}\nu)
\end{equation}
for $\lambda\in \PP(l)$, $\mu\in \PP(m)$, $\nu\in
\PP(n)$, $1\leq s \leq l$ and $1\leq t\leq m$.
Write $\lambda=(a_{l-1}, i)$, $\mu=(a_{m-1}, j)$ and $\nu=(a_{n-1},k)$
as before. If the LHS of \eqref{E7.1.5} is nonzero, then we must have
both $s=i$ and $t=j$. Similarly, if the RHS of \eqref{E7.1.5} is nonzero,
then $s=i$ and $t=j$. So it suffices to consider the case
when $s=i$ and $t=j$. In this case, both sides of \eqref{E7.1.5}
are equal to $(a_{l-1}a_{m-1}a_{n-1},i+j+k-2)$. Hence
\eqref{E7.1.5} holds.

Now we have verified all axioms of a symmetric operad for $\PP:=\SO_A$
and therefore $\SO_A$ is a symmetric operad.

It is obvious that
\begin{equation}
\label{E7.1.6}\tag{E7.1.6}
G_{\SO_A}(z)=z(zH_A(z))'.
\end{equation}
\end{construction}

\begin{lemma}
\label{xxlem7.2}
Let $A$ be a connected graded algebra and let $\SO_A$ be the
symmetric operad provided in Construction \ref{xxcon7.1}.
\begin{enumerate}
\item[(1)]
If $A$ is finitely generated, then $\SO_A$ is finitely generated
as a symmetric operad (resp. as a nonsymmetric operad).
\item[(2)]
If $A$ is finitely presented, then $\SO_A$ is finitely presented
as a nonsymmetric operad.
\item[(3)]
If $A$ is finitely presented, then $\SO_A$ is finitely presented
as a symmetric operad.
\end{enumerate}
\end{lemma}

\begin{proof}
We continue to use the notation introduced in Construction 
\ref{xxcon7.1}.

(1) Let $\fm=\oplus_{i\geq 1} A_i$. Suppose $V\subseteq \fm$ is a 
finite dimensional graded subspace that generates $A$. Let
$\{v_1,\cdots, v_e\}$ be a basis of $V$. We claim that $\SO_A$
is generated by
$E:=\sum_{i=1}^{e} (\sum_{j=1}^{\deg (v_i)+1} \FF (v_i, j))$
where $\deg(v_i)$ is as defined in the graded algebra $A$.
Note that $\Ar((v_i,j))=\deg(v_i)+1$. Since every element in
$\SO_A$ is of the form $(a_{n-1}, d)$ where $1\leq d\leq n$ and
$a_{n-1}$ is generated by $V$ in $A$, it can be generated
by $(v_i,j)$ by partial compositions. Therefore $\SO_A$ is finitely
generated as a nonsymmetric operad (resp. as a symmetric operad).

(2) Suppose $A$ is generated by a finite dimensional graded 
subspace $V$ and subject to a finite dimensional relation 
subspace $W:=\sum_{j=1}^w \FF f_j$. Then $A$ is the factor
algebra $\FF\langle V\rangle/(W)$ where $\FF\langle V\rangle$
is the free algebra generated by $V$ and where $(W)$ denotes the 
relation ideal generated by $W$. Note that $\FF\langle V\rangle$
is connected graded. Let $E$ be defined as in the proof of part 
(1) and let ${\mathcal F}^{ns}(E)$ be the free nonsymmetric
operad generated by $E$ (see \cite[5.9.6]{LV12} for the definition).

Our first step is to define the relation subspace $R$ of the 
${\mathcal F}^{ns}(E)$ such that $\SO_A\cong {\mathcal F}^{ns}(E)/(R)$.
For each homogeneous element $g$ in the free algebra 
$\FF\langle V\rangle$, fix an {\it expression} of $g$ as a 
linear combination of possibly repeated monomials
\begin{equation}
\label{E7.2.1}\tag{E7.2.1}
g=\sum
c_{\bullet} \; v_{i_1} \cdots v_{i_s} \cdots v_{i_u}
\end{equation}
where the sum is over $\bullet:=(i_1,\cdots,i_s, \cdots, i_{u})$
and where $c_{\bullet}$ are scalars in $\FF$.
Note that monomials $\{v_{i_1} \cdots v_{i_s} \cdots v_{i_u}\}$ in 
expression \eqref{E7.2.1} are not assumed to be distinct. For each 
term $v_{i_1} \cdots v_{i_s} \cdots v_{i_u}$ appearing in 
\eqref{E7.2.1}, let $\{k_{i_s}\}_{s=1}^u$ be a sequence of integers such 
that $1\leq k_{i_s}\leq \deg(v_{i_s})+1$ for all $1\leq s\leq u$. 
We will use $\bfk_{\bullet}$ to denote $\{k_{i_s}\}_{s=1}^u$. Define
$$|\bfk_{\bullet}|:=\sum_{s=1}^u k_{i_s}-u+1.$$
Fix any integer $d$ between 1 and $\deg (g)+1$. For each
$\bullet:=(i_1,\cdots,i_u)$ appeared in \eqref{E7.2.1},
pick any sequence $\{k_{i_s}\}_{s=1}^u$ such that $|\bfk_{\bullet}|
=d$. Such a sequence $\{k_{i_s}\}_{s=1}^u$ is called a
{\it $d$-sequence}. Let $\bfk_d$ be the collection of 
$d$-sequences associated to expression 
\eqref{E7.2.1}. Now we define the following element in the 
free operad ${\mathcal F}^{ns}(E)$
\begin{equation}
\label{E7.2.2}\tag{E7.2.2}
r_{\bfk_d}(g):=\sum_{\bfk_{\bullet}\in \bfk_d}
c_{\bullet} \; (v_{i_1}, k_{i_1})
\circ_{k_{i_1}} (v_{i_2}, k_{i_2})
\circ_{k_{i_2}} \cdots \circ_{k_{i_{u-1}}}
(v_{i_u}, k_{i_u})
\end{equation}
where each $(v_{i_1}, k_{i_1})
\circ_{k_{i_1}} (v_{i_2}, k_{i_2})
\circ_{k_{i_2}} \cdots \circ_{k_{i_{u-1}}}
(v_{i_u}, k_{i_u})$ is a right normal tree monomial
defined in \eqref{E4.0.1}-\eqref{E4.0.2}.

For example, let
$$h= v_1 v_4 -2 v_2 v_3 +3 v_1 v_4$$
be in $\FF\langle V \rangle$ where $\deg v_i=i$ for $1\leq i\leq 4$. 
Note that the first monomial in the above expression of $h$ equals 
the third one. Let $d=4$ and we pick three $d$-sequences $\{1, 4\}$,
$\{3,2\}$ and $\{2,3\}$ corresponding to three monomials in 
the expression of $h$. Then 
$$r_{\bfk_{4}}(h)=(v_1, 1)\circ_1 (v_4, 4)
-2 (v_2,3)\circ_3 (v_3,2)+3 (v_1,2)\circ_2 (v_4,3).$$

Let $g'$ be another homogeneous element of degree equal to
$\deg(g)$ and fix an expression of $g'$ similar to 
\eqref{E7.2.1}. Let $c$ and $c'$ be scalars in $\FF$ and 
let $f=cg +c' g'$ with expression of $f$ induced by 
the expression of $g$ and $g'$. Let $\bfk_{d}$ (resp. 
$\bfk'_{d}$) be a collection of $d$-sequences of $g$ 
(resp. $g'$) corresponding to the monomials appeared 
in \eqref{E7.2.1}. Then the disjoint union 
$\bfk_{d} \biguplus \bfk'_d$ is a collection of of 
$d$-sequences corresponding to the expression of $f$. 
It follows from the definition that
\begin{equation}
\label{E7.2.3}\tag{E7.2.3}
r_{\bfk_{d} \biguplus \bfk'_d}(f)
=c r_{\bfk_{d}}(g)+c'r_{\bfk'_d}(g').
\end{equation}

For any given $d$, there are only finitely many possibilities 
for $d$-sequences $\{k_{i_s}\}_{s=1}^u$, and consequently, for
$r_{\bfk_d}(g)$. 

Now let $f_i$ be an element in the relation space $W$ and fix an 
expression of $f_i$ as in \eqref{E7.2.1}. It is easy 
to check that all of 
\begin{equation}
\label{E7.2.4}\tag{E7.2.4}
r_{\bfk_d}(f_j)=0
\end{equation}
are relations of $\SO_A$ for $1\leq d\leq \deg f_j$ and different 
choices of $d$-sequences. By the definition of $\SO_A$, the 
following are also relations of $\SO_A$:
\begin{equation}
\label{E7.2.5}\tag{E7.2.5}
(v_{i_1}, j_1)\circ_l (v_{i_2}, j_2)=0
\end{equation}
for all $(v_{i_s}, j_s) \in E$ and for all $l\neq j_1$
and
\begin{equation}
\label{E7.2.6}\tag{E7.2.6}
(v_{i_1}, j_1)\circ_{j_1} (v_{i_2}, j_2)
-(v_{i_1}, j_1-1)\circ_{j_1-1} (v_{i_2}, j_2+1)=0
\end{equation}
for all $(v_{i_s}, j_s) \in E$ and for all $1<j_1\leq \deg v_{i_1}+1$ and
$1\leq j_2< \deg(v_{i_2})+1$.
Let $R$ be the graded vector subspace generated by the left-hand side 
of equations \eqref{E7.2.4}, \eqref{E7.2.5}, and \eqref{E7.2.6}. 

By part (1) there is a surjective morphism of nonsymmetric operads
\begin{equation}
\notag
\Phi: {\mathcal F}^{ns}(E)\to \SO_A
\end{equation}
sending $(v_i,j)\in E\subseteq {\mathcal F}^{ns}(E)$ 
to $(v_i,j)\in \SO_A$. It is clear that $\Phi$ maps 
relations defined in \eqref{E7.2.4}-\eqref{E7.2.6} 
to zero. Therefore $\Phi$ induces naturally a 
surjective morphism of nonsymmetric operads
\begin{equation}
\label{E7.2.7}\tag{E7.2.7}
\phi: {\mathcal F}^{ns}(E)/(R)\to \SO_A,
\end{equation}
such that $\Phi=\phi\circ \pi$ where we denote $\pi$ the 
canonical quotient map ${\mathcal F}^{ns}(E)
\to {\mathcal F}^{ns}(E)/(R)$.

Our next step is to show that the natural map $\phi$ 
is injective (consequently, bijective). 

We need the following notion of leading form.
For each tree monomial $t\in{\mathcal F}^{ns}(E)/(R)$ of the form
$$t=(v_{i_1}, j_1) \circ_{j_1} (v_{i_2}, j_2)
\circ_{j_2} \cdots \circ_{j_{m-1}}(v_{i_m}, j_{m}),$$
there exist finitely many elements of the form 
$$t'=(v_{i_1}, j'_1) \circ_{j'_1} (v_{i_2}, j'_2)
\circ_{j'_2} \cdots \circ_{j'_{m-1}}
(v_{i_m}, j'_{m}),$$
such that $ j'_1+\cdots +j'_m=j_1+\cdots +j_m$ and 
$1\leq j'_l\leq \deg(v_{i_l})+1$ for $1\leq l\leq m$. 
Call the largest tree monomial (under the path-lexicographic order) 
among these elements the {\it leading form} of $t$, denote it by 
$L(t)$. Let $(R_{7.2.6})$ be the ideal of 
${\mathcal F}^{ns}(E)$ generated by relations given in 
\eqref{E7.2.6}. It follows from \eqref{E7.2.6} that 
$L(t)-t\in(R_{7.2.6})$ and it is clear that $L(t')-L(t)=0$ 
for any choice of $(j'_1,\cdots,j'_{m})$. As a consequence,
modulo the ideal $(R_{7.2.6})$, $r_{\bfk_d}(g)$ 
is independent of the expressions of $g$ given in \eqref{E7.2.1} 
and the choices of $\bfk_{d}$ -- the collection of 
$d$-sequences. 
By pre-composing with $\pi$, the argument in this paragraph 
together with \eqref{E7.2.3} implies that $\pi r_{\bfk_d}(g)$ 
is independent of the expressions of $g$ 
and the choices of $d$-sequences and that $\pi r_{\bfk_{d}}$
is a well defined $\FF$-linear map from $\FF\langle V\rangle
\to {\mathcal F}^{ns}(E)/(R)$.

Fix any $n$, let $\{a_\alpha\}_{\alpha\in I_{n}}$ be a monomial 
basis of $A_{n-1}$ for some index set $I_n$. Then 
$\{(a_\alpha,i)\}_{\alpha\in I_n, 1\leq i\leq n}$ is an $\FF$-linear 
basis of $\SO_A(n)$. For any 
$a_\alpha=v_{\alpha,1}v_{\alpha,2}\cdots v_{\alpha,m_\alpha}$, 
let $b'_{\alpha,i}\in{\mathcal F}^{ns}(E)/(R)$ be a monomial 
of the following form
$$b'_{\alpha,i}=(v_{\alpha,1},j'_1)\circ_{j'_1}
(v_{\alpha,2},j'_2)\circ_{j'_2}\cdots \circ_{j'_{m_\alpha-1}}
(v_{\alpha,m_\alpha},j'_{m_\alpha}),$$
where $j'_1+\cdots+j'_{m_\alpha}=i+m_\alpha-1$ and 
$1\leq j'_l\leq \deg(v_{\alpha,l})+1$ for $1\leq l\leq m_\alpha$. 
Then $b'_{\alpha,i}\in\phi^{-1}((a_\alpha,i))$. Take $b_{\alpha,i}$ 
to be the largest one among the monomials of this form, i.e.,  
$b_{\alpha,i}:=L(b'_{\alpha,i})$, which is independent of the 
choice of $b'_{\alpha,i}$. By definition, both $b_{\alpha,i}$
and $b'_{\alpha,i}$, considered as elements in ${\mathcal F}^{ns}(E)$, 
are of the form $r_{\bfk_{i}}(a_{\alpha})$ for some choices 
of $i$-sequences $\bfk_{\bullet}$. By \eqref{E7.2.6}, 
$b_{\alpha,i}=b'_{\alpha,i}$ in ${\mathcal F}^{ns}(E)/(R)$. 
The conclusion of this paragraph is that $b_{\alpha,i}
=\pi r_{\bfk_i}(a_{\alpha})$ for $1\leq i\leq n$ and any
collection of $i$-sequences.

By definition, any monomial $s\in{\mathcal F}^{ns}(E)/(R)(n)$ (or
in ${\mathcal F}^{ns}(E)(n)$) is left normal, namely,
$$s=(\cdots((v_{i_1}, j_1) \circ_{k_1} (v_{i_2}, j_2))
\circ_{k_2} \cdots) \circ_{k_{m-1}}
(v_{i_m}, j_{m}).$$
Using the relations of the form \eqref{E7.2.5}, $s$ in 
${\mathcal F}^{ns}(E)/(R)(n)$ is either zero or equal to
\begin{equation}
\label{E7.2.8}\tag{E7.2.8}
s=(v_{i_1}, j_1) \circ_{j_1} (v_{i_2}, j_2)
\circ_{j_2} \cdots\circ_{j_{m-1}} (v_{i_m}, j_{m}),
\end{equation}
namely, it is right normal, see \eqref{E4.0.1}-\eqref{E4.0.2}.
For the rest of the proof of part (2), we will only 
use right normal tree monomials in ${\mathcal F}^{ns}(E)$ 
and these will simply be called {\it monomials}. For each 
monomial $s$, we will freely replace $s$ by $L(s)$ and 
vice versa as $L(s)=s$ in ${\mathcal F}^{ns}(E)/(R)$ by 
\eqref{E7.2.6}.

To prove $\phi$ is injective, it suffices to show that 
$$
\dim_{\FF}(\SO_A(n))\geq \dim_{\FF} {\mathcal F}^{ns}(E)/(R)(n)\ \ 
\text{ for all } n.
$$
Then it is enough to show that every $s$ 
in \eqref{E7.2.8} can be presented as a linear combination of 
$\{b_{\alpha,i}\}_{\alpha\in I_n, 1\leq i\leq n}$ in 
${\mathcal F}^{ns}(E)/(R)(n)$. 

Since $\Ar((v_{i_1}, j_1))+\cdots+\Ar((v_{i_m}, j_{m}))=n+m-1$, 
the element $v_{i_1}v_{i_2}\cdots v_{i_m}\in A_{n-1}$ and it 
can be presented in the free algebra generated by $V$ as follows:
\begin{equation}
\notag
v_{i_1}v_{i_2}\cdots v_{i_m}=
\sum_\alpha c_\alpha a_\alpha+\sum_\gamma c_\gamma v_{\gamma,1}
\cdots v_{\gamma,p_\gamma} g_\gamma v_{\gamma,p_\gamma+1}
\cdots v_{\gamma,q_\gamma}
\end{equation}
where $c_\alpha,c_\gamma\in\FF$ and $g_\gamma\in \{f_{j}~|~1\leq j\leq w\}$. 
Rewrite the above equation as
\begin{equation}
\label{E7.2.9}\tag{E7.2.9}
v_{i_1}v_{i_2}\cdots v_{i_m}-
\sum_\alpha c_\alpha a_\alpha=\sum_\gamma c_\gamma v_{\gamma,1}
\cdots v_{\gamma,p_\gamma} g_\gamma v_{\gamma,p_\gamma+1}
\cdots v_{\gamma,q_\gamma}.
\end{equation}
We claim that $r_{\bfk_d}(g)\in (R)$, or equivalently,
$\pi r_{\bfk_d}(g)=0$, where $g$ is the right-hand side 
of \eqref{E7.2.9}. Note that each monomial in right-hand side 
of \eqref{E7.2.9} has degree $n-1$. If the claim holds, then, 
for $d:=\sum_{w=1}^m j_w-m+1$,
$$s=\pi r_{\bfk_d}(v_{i_1}v_{i_2}\cdots v_{i_m})
=\pi r_{\bfk_d}(\sum_{\alpha} c_{\alpha} a_{\alpha})
=\sum_{\alpha} c_{\alpha} \pi r_{\bfk_d}(a_{\alpha})
=\sum_{\alpha} c_{\alpha} b_{\alpha,d}$$
in ${\mathcal F}^{ns}(E)/(R)$.
Now we prove the claim. Since $\pi r_{\bfk_d}(-)$ is additive, 
we may assume that $g$ only has one term, namely,
\begin{equation}
\label{E7.2.10}\tag{E7.2.10}
g=v_{1}\cdots v_{p} f v_{p+1}\cdots v_{q}
\end{equation}
where $f$ is one of $f_i$ in $W$. Write $f$ as a linear 
combination of monomials, say $\{v_{f,1}\cdots v_{f,m_f}\}$. 
Then \eqref{E7.2.10} can be considered as an expression of
$g$ which is a linear combination of monomials of the form 
\begin{equation}
\label{E7.2.11}\tag{E7.2.11}
s_h:=v_{1}\cdots v_{p} v_{f,1}\cdots 
v_{f,m_f} v_{p+1}\cdots v_{q}.
\end{equation}  
For any $d$-sequence 
$\bfk_{\bullet}$ corresponding to \eqref{E7.2.11}, 
it can be decomposed into $\bfk(b)_{\bullet}:=\{k_{l}\}_{i=1}^p$ 
and $\bfk(m)_{\bullet}:=\{k_{f,l}\}_{l=1}^{m_f}$, and 
$\bfk(e)_{\bullet}:=\{k_{l}\}_{k=p+1}^q$
such that $1\leq k_{l}\leq \deg(v_{l})+1$ for $1\leq l\leq q$, 
$1\leq k_{f,l}\leq \deg(v_{f,l})+1$ for $1\leq l\leq m_f$, 
and 
$$\sum_{l=1}^{q}k_{l}+\sum_{l=1}^{m_f}k_{f,l}-q-m_f+1=d.$$
Note that $\bfk(b)_{\bullet}$ (resp. $\bfk(f)_{\bullet}$, 
$\bfk(e)_{\bullet}$) is the beginning part (resp. the middle part, 
the ending part) of $\bfk_{\bullet}$. For different $s_h$, since
$\deg v_{f,1}\cdots v_{f,m_f}=\deg f$ which is independent of the 
individual monomial $v_{f,1}\cdots v_{f,m_f}$, one can easily 
choose $\bfk_{\bullet}$ such that $\bfk(b)_{\bullet}$ and 
$\bfk(e)_{\bullet}$ are independent of the middle part 
$v_{f,1}\cdots v_{f,m_f}$. Therefore 
$$r_{\bfk_d}(s_h)=r_{\bfk(b)_{d_b}}(v_{1}\cdots v_{p})\circ_{d_b}
r_{\bfk(f)_{d_f}}(v_{f,1}\cdots v_{f,m_f}) \circ_{d_f}
r_{\bfk(e)_{d_e}}(v_{p+1}\cdots v_{q})$$
for some fixed $\bfk(b)_{d_b}$, $\bfk(e)_{d_e}$ and $d_f$. 
This fact implies that 
$$r_{\bfk_d}(g)=r_{\bfk(b)_{d_b}}(v_{1}\cdots v_{p})\circ_{d_b}
r_{\bfk(f)_{d_f}}(f)\circ_{d_f}
r_{\bfk(e)_{d_e}}(v_{p+1}\cdots v_{q})
\in (R).$$
So we proved that $r_{\bfk_d}(g)\in (R)$ as desired.

(3) By part (1), there is a surjective morphism of 
symmetric operads
$$\Psi: {\mathcal F}^{sy}(E)\to \SO_A$$
where ${\mathcal F}^{sy}(E)$ is the free symmetric operad
generated by $E$. Let $R$ be the relation subspace defined
in the proof of part (2). It is clear that 
$\Psi$ maps $R$ to 0. Hence $\Psi$ induces naturally 
a surjective morphism of symmetric operads
$$\psi: {\mathcal F}^{sy}(E)/(R\ast \SG) \to \SO_A$$
where $R\ast \SG$ is the space generated by $f\ast \sigma$
for all $f\in R$ and $\sigma\in \SG_{\Ar(f)}$.
By the universal property, there is a morphism of 
nonsymmetric operads ${\mathcal F}^{ns}(E)\to
{\mathcal F}^{sy}(E)/(R\ast \SG)$ which induces a
morphism of nonsymmetric operads ${\mathcal F}^{ns}(E)/(R)\to
{\mathcal F}^{sy}(E)/(R\ast \SG)$. So we have the following
sequence of morphisms
$${\mathcal F}^{ns}(E)/(R)\xrightarrow{\theta} 
{\mathcal F}^{sy}(E)/(R\ast \SG)\xrightarrow{\psi} \SO_A.$$
By part (2), the composition is an isomorphism and consequently 
$\theta$ is injcetive. For simplicity, we  consider $\theta$ as an 
inclusion and identity 
$f\in {\mathcal F}^{ns}(E)/(R)$ with $\theta(f)
\in {\mathcal F}^{sy}(E)/(R\ast \SG)$.

It follows from the Equivariance Axiom \cite[Definition 1.2(OP3')]{BYZ20},
every element in ${\mathcal F}^{sy}(E)/(R\ast \SG)$ is a linear
combination of elements of the form $s \ast \sigma$ where 
$s\in {\mathcal F}^{ns}(E)/(R)$ and $\sigma\in \SG$. It remains
to show the claim that every element of the form $s \ast \sigma$ is 
in fact in ${\mathcal F}^{ns}(E)/(R)$. By the proof of part (2), we
may further assume that
$$s=(v_{i_1}, j_1) \circ_{j_1} (v_{i_2}, j_2)
\circ_{j_2} \cdots \circ_{j_{m-1}}(v_{i_m}, j_{m})$$
where $v_{i_s}\in E$ and $1\leq j_s\leq \deg(v_{i_s})+1$.
Let $d$ be $\sum_{s=1}^m j_s-m+1$. Note that $\{j_s\}_{s=1}^m$
can be replaced by any $d$-sequence in the above formula.
Let $\sigma$ be any permutation in $\SG_{\Ar(s)}$. Using induction 
on $m$, relations of form \eqref{E7.2.6}, and the fact
$$(v_{i_s}, j)\ast \tau =(v_{i_s}, \tau^{-1}(j)$$
for all $\tau\in \SG_{\Ar(v_{i_s},j)}$ we obtain that
$$s\ast \sigma =(v_{i_1}, j'_1) \circ_{j'_1} (v_{i_2}, j'_2)
\circ_{j'_2} \cdots \circ_{j'_{m-1}}(v_{i_m}, j'_{m})$$
where $d':=\sum_{s=1}^m j'_s-m+1$ is equal to $\sigma^{-1}(d)$. 
(The above equation can also be seen from the tree presentation 
of $s$ and $s\ast \sigma$). Again $\{j'_s\}_{s=1}^m$ can be 
any $d'$-sequence by relations \eqref{E7.2.6}. Hence 
$s\ast \sigma\in {\mathcal F}^{ns}(E)/(R)$ as desired.
\end{proof}

Note that Construction \ref{xxcon7.1} can be viewed as a
symmetric version of Construction \ref{xxcon1.3}. We are
wondering if there is a symmetric version of Construction
\ref{xxcon5.1}. If $A$ is a commutative graded algebra,
there is another construction of a symmetric operad
associated to $A$, see \cite[Section 4.2]{DK10}.

Suppose a symmetric operad $\PP$ is finitely generated by a finite
alphabet $\XX$. Let $\VV=\FF \XX$, and for every $m\geq 0$,
$\VV^{m}$ is defined as in \eqref{E3.1.1}. For any subcollection
$\WW \subseteq \PP$, let $\WW_{\SG}=\{\WW(n)\otimes \SG_n\}_{n\geq 0}$.
Now let $\VSG^{m}$ denote $(\VV^{m})_{\SG}$, so
$\VSG^{m}(n)$ is a right $\SG_n$-module for all $n$. We have a symmetric
version of Lemmas \ref{xxlem3.2} and \ref{xxlem3.3}.

\begin{lemma}
\label{xxlem7.3}
Suppose $\PP$ is a locally finite symmetric operad generated
by a finite dimensional subcollection $\VV$.
\begin{enumerate}
\item[(1)]
Then
\[
\GK(\PP)=\limsup_{n\to\infty}\log_n
\left(\dim\left(\sum_{i=0}^n \VSG^i\right)\right).
\]
\item[(2)]
The reduced symmetric operad associated with $\PP$, defined as
in Lemma \ref{xxlem3.2}(1), is finitely generated and locally
finite.
\item[(3)]
The reduced connected symmetric operad associated with $\PP$,
defined as in Lemma \ref{xxlem3.2}(2) is finitely generated and
locally finite.
\end{enumerate}
\end{lemma}

We have a version of Proposition \ref{xxpro3.8} for
symmetric operads.

\begin{lemma}
\label{xxlem7.4}
Suppose $\PP$ is a finitely generated locally finite symmetric
operad.
\begin{enumerate}
\item[(1)]
$\GK(\PP)=0$ if and only if $\PP$ is finite dimensional.
\item[(2)]
$\GK(\PP)$ cannot be strictly between 0 and 1.
\end{enumerate}
\end{lemma}

Now we are ready to show Theorem \ref{xxthm0.7}.

\begin{proof}[Proof of Theorem \ref{xxthm0.7}]
(1) Let $\PP$ be a finitely generated and locally finite
symmetric operad. It is clear that $\GK(\PP)\geq 0$.
The assertion follows from Lemma \ref{xxlem7.4}(2).

(2) It is well-known that, if $r\in \NN$, then there are
finitely generated and locally finite symmetric operads
$\PP$ such that $\GK(\PP)=r$. Now suppose that
$r$ is not an integer and $r\in R_{\GK} \setminus (2,3)$.
Then $r>3$. By Lemma \ref{xxlem3.6}(3), there is a 
connected graded
algebra $A$ such that $\GK(A)=r-1$ and $f(n):=\dim A_n$ is
increasing. By Construction \ref{xxcon7.1} and
Lemma \ref{xxlem7.2}, there is a locally
finite and finitely generated symmetric operads $\SO_A$
such that $\dim \SO_A(n)=n f(n-1)$ for all $n\in \NN$.
Then, using the fact that $f(n)$ are increasing,
$$\begin{aligned}
\GK(\SO_A)&=\limsup_{n\to\infty}
\log_n (\sum_{i=0}^n\dim \SO_A(i))=\limsup_{n\to\infty}
\log_n (\sum_{i=0}^n i f(i-1))\\
&\leq \limsup_{n\to\infty}
\log_n (n \sum_{i=0}^n f(i-1))
=1+\GK(A),\\
\GK(\SO_A) &=\limsup_{n\to\infty}
\log_n (\sum_{i=0}^n i f(i-1))
\geq \limsup_{n\to\infty}
\log_n (\sum_{i=\lfloor n/2\rfloor+1}^n i f(i-1))\\
&\geq \limsup_{n\to\infty}
\log_n (\frac{n}{2} \sum_{i=\lfloor n/2\rfloor+1}^n f(i-1))
\geq \limsup_{n\to\infty}
\log_n (\frac{n}{2} \sum_{i=1}^{\lfloor n/2\rfloor} f(i-1))\\
&=1+\GK A.
\end{aligned}
$$
Therefore $\GK(\SO_A)=\GK(A)+1=r$.

(3) Let $U$ be the finitely generated graded algebra given in
Example \ref{xxex6.2}. So $\GK(U)=1$. Let $\SO_U$ be the
finitely generated locally finite symmetric operad constructed
in Construction \ref{xxcon7.1} and Lemma \ref{xxlem7.2}.
By the proof of part (2),
$\GK(\SO_U)=2$. In fact, its generating series is
$$G_{\SO_U}(z)=z(zH_U(z))'
=\frac{z(1+2 z+2 z^2-2 z^3)}{(1-z)^2} +\bar{V}(z)$$
where
$$\bar{V}(z)=\sum_{n=2}^{\infty} (n+1)\delta_{\Lambda}(n) z^{n+1}.$$
By Lemma \ref{xxlem6.1}, $\bar{V}(z)$ and hence $G_{\SO_U}(z)$
are not holonomic.

Let $r\geq 2$ be any real number such that there is a finitely
generated locally finite symmetric operad $\PP$ with
$\GK(\PP)=r$. If $G_{\PP}$ is not holonomic, then we are done.
Otherwise, $G_{\PP}$ is holonomic. Now consider a new operad
$\QQ:=\PP\oplus \SO_{U}$ with $G_{\QQ}=G_{\PP}+G_{\SO_U}$. By
\cite[Holonomic Theorem 2]{Ber14}, $G_{\QQ}$ is not holonomic.
Since $\GK(\SO_U)=2$, we have $\GK(\QQ)=\GK(\PP)=r$. The assertion
follows.
\end{proof}

We finish the paper with a potential counterexample to
\cite[Expectation 1]{KP15}.

\begin{example}
\label{xxex7.5}
Let $A$ be the connected graded algebra in Example \ref{xxex6.4}
and let $\SO_A$ be the operad given in Construction \ref{xxcon7.1}.
Then $\SO_A$ is a finitely presented symmetric operad by
Lemma \ref{xxlem7.2}. By
\eqref{E7.1.6}, its generating series is
$$G_{\SO_A}(z)=z(z H_A(z))'=z(zP(z))'$$
which is not holonomic by Property
(P3) in Example \ref{xxex6.4}. Therefore $\SO_A$ is a
``non-generic'' counterexample to the symmetric version of
\cite[Expectation 2]{KP15}.

Since $G_{\SO_A}(z)$ is not holonomic, by \eqref{E6.0.3},
$E_{\SO_A}(z)$ is not holonomic. We conjecture that
$E_{\SO_A}(z)$ is not differential algebraic. If this is the case,
then $\SO_A$ is a
``non-generic'' counterexample to \cite[Expectation 1]{KP15}.
\end{example}


\section*{Acknowledgments}
The authors thank the referee for his/her careful reading and
valuable comments. The authors thank Yanhong Bao, Yu Li, Yu Ye, and
Zerui Zhang for reading an earlier version of the paper, thank
Jason Bell for pointing out the fact that the partition
generating function is not holonomic and for providing several
references, and thank Vladimir Dotsenko for email exchanges
and for his interesting remarks on the subject. Z.-H. Qi,
Y.-J. Xu and X.-G. Zhao would like to thank J.J. Zhang
and the department of Mathematics at the University of Washington
for the hospitality during their visits. Z.-H. Qi was partially
supported by National Science Foundation of China (Nos. 11971460, 12071137) and Science and Technology Commission of Shanghai Municipality (No. 18dz2271000). Y.-J. Xu was partially
supported by National Science Foundation of China (Nos. 11501317,
11871301), and China Postdoctoral
Science Foundation (No. 2016M600530).
J.J. Zhang was partially supported by the US National
Science Foundation (Nos. DMS-1700825 and DMS-2001015).
X.-G. Zhao was partially supported by the Natural
Science Foundation of Huizhou University (HZU202001, HZU201804)
and the Characteristic Innovation Project of Guangdong Provincial
Department of Education (2020KTSCX145).

\bibliographystyle{amsalpha}

\end{document}